\documentclass[10pt]{amsart}
\usepackage{amssymb,amsmath}
\usepackage{hyperref}
\usepackage{bm}
\usepackage{graphicx}
\usepackage{amssymb, epsfig}
\newtheorem{theorem}{Theorem}[section]
\newtheorem{lemma}[theorem]{Lemma}
\newtheorem{proposition}[theorem]{Proposition}
\theoremstyle{definition}

\theoremstyle{remark}
\newtheorem{remark}[theorem]{Remark}
\numberwithin{equation}{section}
\newcommand{\nn}{n+\frac{1}{2}}

\newcommand{\rnn}{r^{n+\frac{1}{2}}}
\newcommand{\rnnm}{r^{n-\frac{1}{2}}}

\newcommand{\az}{\theta}

\newcommand{\szm}{\partial_{y=0}}

\newcommand{\ste}{\partial_{\theta=\theta_1}}
\newcommand{\std}{\partial_{\theta=\theta_2}}
\newcommand{\HH}{\tilde{H}_0^1(\mathfrak{D})}

\newcommand{\om}{\mathcal{W}}
\newcommand{\up}{\mathcal{U}}
\newcommand{\vp}{\mathcal{V}}
\newcommand{\bl}{\textit{B}}

\begin{document}
\title{Crank-Nicolson Finite Element Discretizations\\
for a 2D Linear Schr{\"o}dinger-Type Equation\\
Posed in a Noncylindrical Domain}
\author{D.~C. Antonopoulou, G.~D. Karali, M. Plexousakis, G.~E. Zouraris}
\address{Department of Applied Mathematics, University of
Crete, 714 09 Heraklion, Greece, and Institute of
Applied and Computational Mathematics, FORTH, Greece.}
\email{danton@tem.uoc.gr}
\email{gkarali@tem.uoc.gr}
\email{plex@tem.uoc.gr}

\address{Department of Mathematics,
University of Crete, 714 09 Heraklion, Greece.}
\email{zouraris@math.uoc.gr}
%
%
\begin{abstract}
Motivated by the paraxial narrow--angle approximation of the
Helm\-holtz equation in domains of variable topography that
appears as an important application in Underwater Acoustics, we
analyze a general Schr\"odinger-type equation posed on
two-dimensional variable domains with mixed boundary conditions.
The resulting initial- and boundary-value problem is transformed
into an equivalent one posed on a rectangular domain and is
approximated by fully discrete, $L^2$-stable, finite element,
Crank--Nicolson type schemes. We prove a global elliptic
regularity theorem for complex elliptic boundary value problems
with mixed conditions and derive $L^2$-error estimates of optimal
order. Numerical experiments are presented which verify the optimal
rate of convergence.
\end{abstract}
\maketitle
\section{Introduction}
\subsection{The physical problem}
The standard narrow-angle Parabolic Equation
(PE) in three space dimensions is the following
Schr\"odinger-type equation
\begin{equation}\label{NA}
\psi_{r}=\tfrac{\mathrm{i}}{2\,k_0}\,\left(\,\psi_{zz} +
\tfrac{1}{r^2}\,\psi_{\theta\theta}\,\right)
+\mathrm{i}\,\tfrac{k_0}{2}\,(n^2-1)\,\psi,
\end{equation}
that models the long-range sound propagation in the sea, and is
used in the context of underwater acoustics as the paraxial and
far-field approximation of the Helmholtz equation in the presence
of cylindrical symmetry, cf. \cite{ref54,ref12}. Here, $r_{\rm
max}\ge r\ge r_{\rm min}>0$ is the horizontal distance from a
harmonic point source placed on the $z$ axis and emitting at a
frequency $f_0$. The function $\psi=\psi(r,z,\theta)$ depending on
range, depth and azimuth measures the acoustic pressure in
inhomogeneous, weakly range-dependent marine environments. The
depth variable $z\ge 0$ is increasing downwards while the azimuth
varies in the interval $[\theta_{\rm min},\theta_{\rm max}]$;
$k_0=\tfrac{2\,\pi\,f_0}{c_0}$ is a reference wave number, the constant $c_0$
is a reference sound speed, $n(r,z,\theta) = c_0/c(r,z,\theta)$
is the refraction index and $c(r,z,\theta)$ is the sound speed in
the water. The bottom topography, being variable, is identified
in cylindrical coordinates by a positive surface $z =
s(r,\theta)$.
\par
For a fixed range $r\in[r_{\rm min},r_{\rm max}]$, we define the
$r$-dependent space domain:
\begin{equation*}
\Omega(r):=\Big{\{}(z,\theta)\in
\mathbb{R}^2:\;\theta\in[\theta_{\rm min},\;\theta_{\rm
max}],\;z\in[0,s(r,\theta)]\Big{\}},
\end{equation*}
where obviously, $\partial
\Omega(r)=\displaystyle{\cup_{i=1}^4}\omega_i(r)$ for
$\omega_1(r):=\{(0,\theta)\in \mathbb{R}^2:\;\theta
\in[\theta_{\rm min},\theta_{\rm max}]\}$,
$\omega_2(r):=\{(z,\theta_{\rm min})\in
\mathbb{R}^2:\;z\in[0,s(r,\theta_{\rm min})]\}$,
$\omega_3(r):=\{(s(r,\theta),\theta)\in
\mathbb{R}^2:\;\theta\in[\theta_{\rm min},\theta_{\rm max}]\}$,
and $\omega_4(r):=\{(z,\theta_{\rm max})\in
\mathbb{R}^2:\;z\in[0,s(r,\theta_{\rm max})]\}$ (cf.
Figure~\ref{rdomain}).
\begin{figure}[ht]
\centering
\includegraphics[height=4.4truecm, width=9.0truecm]{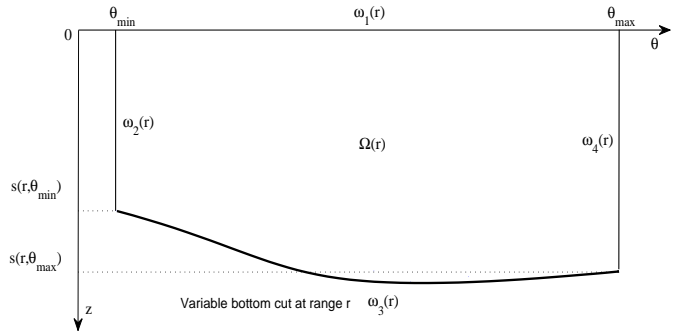}
\caption{The range dependent domain $\Omega(r)$.}\label{rdomain}
\end{figure}
\par
The horizontal sea surface of the naval environment is assumed to
be perfectly absorbing, so a free--release condition $\psi=0$ is
imposed on $\omega_1(r)$. We also set $\psi=0$ on the minimum and
maximum azimuthal values i.e. at $\omega_2(r)\cup\omega_4(r)$. We
denote by
$\omega_D(r):=\omega_1(r)\cup\omega_2(r)\cup\omega_4(r)$ the
piecewise linear boundary segment where these homogeneous Dirichlet
conditions are imposed. The acoustically rigid bottom is
mathematically modeled by the Neumann boundary condition
$\frac{\partial \psi}{\partial \eta_s}=0$ along the bottom
surface $z=s(r,\cdot)$, i.e., the variable boundary segment
$\omega_3(r)$ of $\Omega(r)$. Even in one space dimension,
the well-posedness of the standard narrow--angle Parabolic
Equation \eqref{NA} with Neumann condition was proved under the
assumption that the bottom topography is strictly monotone, cf.
\cite{ref1}. Considering the same problem, in \cite{ref2,AnDZ}
the authors verified numerically that significant instabilities
develop even in strictly monotone downsloping bottom profiles.

Abrahamsson and Kreiss in \cite{ref1,ref2} proposed alternatively
the use of a Robin-type condition as an approximation of the
Neumann one that yields a well-posed initial and boundary value
problem when the domain topography is variable. This approximate
condition in two dimensions has the form (cf. \cite{ref53}):
\begin{equation}\label{A-K}
\psi_{z}-\frac{s_{\theta}}{r^2}\psi_{\theta}=\mathrm{i}k_{0}s_{r}
\psi\quad\text{at}\quad z=s(r,\theta).
\end{equation}
We impose \eqref{A-K} at $z=s$, we set in \eqref{NA}
$a:=\frac{1}{2k_0}$,
$\beta_{\psi}(r,z,\theta):=\frac{k_0}{2}(n^2-1)$ and arrive at the
following initial and boundary value problem (ibvp) of
Schr\"odinger type:
\begin{equation}\label{uapro}
\begin{split}
&\psi_{r}=\mathrm{i\;div}(D_{\rm a}\nabla\psi)+\mathrm{i}\beta_{\psi}\psi\;\;\;\;\;\;\;\;\mbox{in}\;\;\;S,\\
&\psi=0\;\;\;\;\;\;\;\;\;\;\;\;\;\;\;\;\;\;\;\;\;\;\;\;\;\;\;\;\;\;\;\;\;\;\;\;\;\;\;\,\mbox{on}\;\;\;\omega_D(r)\;\;\;\forall
r\in[r_{\rm min},r_{\rm max}],\\
&\eta_{\rm s}^t(D_{\rm a}
\nabla\psi)=\frac{\mathrm{i}}{2}\frac{s_{r}}{\sqrt{1+s_{\theta}^2}}\psi
\;\;\;\;\;\;\;\;\,\mbox{on}\;\;\;\omega_3(r)\;\;\;\forall
r\in[r_{\rm min},r_{\rm max}],\\
&\psi(r_{\rm
min},z,\theta)=\psi_0(z,\theta)\;\;\;\;\;\;\;\;\;\;\;\;\;\,\mbox{on}\;\;\;
\Omega(r_{\rm min}),
\end{split}
\end{equation}
posed on the non-cylindrical domain $S:=\cup_{r\in[r_{\rm
min},r_{\rm max}]}\Omega(r)$. Here, the gradient is with respect
to the $z, \theta$ variables, $D_{\rm a}:=\begin{pmatrix}
  _{a} & _{0} \\
  _{0} & _{a/r^2}
\end{pmatrix}$, $\eta_{\rm s}=-\frac{(-1,s_{\theta})^t}{\sqrt{1+s_{\theta}^2}}$ is the vector normal to the surface
$z=s$ and the initial condition $\psi_0$ models the acoustic
source.
\begin{remark}
{\rm In view of the ibvp \eqref{uapro}, we observe that the same
term $D_{\rm a}\nabla \psi$ appears at the equation as well as at
the left-hand side of the Abrahamsson-Kreiss Robin condition.}
\end{remark}

%
\subsection{Change of variables}
The focus of our interest herein is to write the problem into an
equivalent form posed on a cylindrical domain where simpler
stable numerical schemes can be applied. This is achieved by a
horizontal change of variables combined with an exponential
transformation.
Specifically, we let
\begin{equation}\label{tr1}
\begin{split}
&y = z/s(r,\theta),\;\;\;\; v(r,y,\theta) =
e^{-q(r,\theta)}\psi(r,z,\theta),\\
&\Omega(r)\hookrightarrow\mathfrak{D}:=(0,1)\times(\theta_{\rm
min},\theta_{\rm max}),\\
&S\hookrightarrow[r_{\rm min},r_{\rm max}]\times \mathfrak{D},
\end{split}
\end{equation}
where $q(r,\theta) = -\frac{1}{2}\ln s(r,\theta)$
\cite{ref53,refthes}. With this choice of $q$,
the initial-{} and boundary-value problem \eqref{uapro} takes
the following form (see \cite{refthes} for the details):
\begin{equation}\label{cvpro}
\begin{split}
&v_{r}=\mathrm{i\;div}(\widehat{D}_{\rm a}\nabla
v)+y\frac{s_r}{s}v_y+\mathrm{i}\beta_v v\;\;\;\;\;\;\;\;
\mbox{in}\;\;\;[r_{\rm min},r_{\rm max}]\times \mathfrak{D},\\
&v=0\;\;\;\;\;\;\;\;\;\;\;\;\;\;\;\;\;\;\;\;\;\;\;\;\;\;\;\;\;\;\;\;\;\;\;\;\;\;\;\,\;\;\mbox{at}\;\;\;y=0\;\;\;\forall
(r,\theta)\in[r_{\rm min},r_{\rm max}]\times[\theta_{\rm min},\theta_{\rm max}],\\
&\eta^t(\widehat{D}_{\rm a}\nabla v) =\mathrm{i}\gamma_{\rm bc}v
\;\;\;\;\;\;\;\;\;\;\;\;\;\;\;\;\;\;\;\;\;\;\mbox{at}\;\;\;y=1\;\;\;\forall
(r,\theta)\in[r_{\rm min},r_{\rm max}]\times[\theta_{\rm min},\theta_{\rm max}],\\
&v(r_{\rm min},y,\theta)=v_0(y,\theta)
\;\;\;\;\;\;\;\;\;\;\;\;\;\,\;\;\forall(y,\theta)\in
\overline{\mathfrak{D}},
\end{split}
\end{equation}
where $\eta:=(1,0)^t$, $\gamma_{\rm bc}(r,\theta):=
\frac{\mathrm{1}}{2}\Big{[}\frac{s_{r}}{s}+\mathrm{i}\frac{a}{r^2}\Big{(}\frac{s_{\theta}}{s}\Big{)}^2\Big{]}$,
and $\widehat{D}_{\rm a}:=\begin{pmatrix}
  _{\widehat{a}} & _{\widehat{\beta}} \\
  _{\widehat{\beta}} & _{\widehat{\gamma}}
\end{pmatrix}$
with
\begin{equation*}
\begin{split}
\widehat{a}(r,y,\theta)=\frac{a}{s^2}+\frac{a}{r^2}y^2\Big{(}\frac{s_{\theta}}{s}\Big{)}^2,\;\;
\widehat{\beta}(r,y,\theta)=-\frac{a}{r^2}y\Big{(}\frac{s_{\theta}}{s}\Big{)},\;\;
\widehat{\gamma}(r)=\frac{a}{r^2},
\end{split}
\end{equation*}
$\beta_v:=\beta_{\psi}+\frac{a}{r^2}\frac{3s_{\theta}^2-2ss_{\theta\theta}}{4s^2}-\mathrm{i}\frac{s_r}{2s}$,
and $v_0(y,\theta)=\sqrt{s(r_{\rm min},\theta)}u_0(ys(r_{\rm
min},\theta),\theta)$.

We note that $\widehat{D}_{\rm a}$ is a real, symmetric and
positive definite matrix and therefore ${\rm
det}(\widehat{D}_{\rm a})>0$, \cite{refthes}. Furthermore, due to
the definition of $q$ the coefficient of $v_y$ in the first
equation is a real function, which at $y=1$ equals to $2{\rm
Re}\gamma_{\rm bc}$.

 Certain three-dimensional effects have been
observed to influence the acoustic transmission in variable
domains mainly because the refraction index depends on $r$, $z$,
$\theta$ and since significant reflections may occur between the
bottom and the see surface (cf. \cite{ref33,
ref21,ref56,ref16,stcas}). In \cite{ref53}, F. Sturm considered
the Narrow--angle parabolic equation with the
Abra\-hamsson-Kreiss condition in three dimensions over a
variable bottom in the case of a multilayered fluid medium.

The single layer case in the presence of azimuthal symmetry where
the physical problem is posed on one-dimensional variable domains
has been analyzed rigorously in \cite{ref8,AnDZ}. More
specifically, in \cite{ref8} the authors constructed finite
difference schemes and proved optimal rate of convergence. In
\cite{AnDZ}, error estimates of optimal order in the $L^2$- and
$H^1$-norms have been proved for semidiscrete and fully discrete
Crank-Nicolson-Galerkin finite element approximations.
Discontinuous Galerkin methods for the linear Schr\"odinger
equation Dirichlet problem in non-cylindrical domains of
$\mathbb{R}^m$, $m\ge 1$, were analyzed in \cite{An-P}. When
$m=1$ the resulting problem is the standard Narrow--angle
parabolic approximation modeling an acoustically soft bottom; for
this case the authors investigated theoretically and numerically
the order of convergence using finite element spaces of piecewise
polynomial functions. The Wide--angle parabolic equation consists
an alternative approximation model of Helmholtz equation in
underwater acoustics; for a rigorous numerical analysis and
numerical experiments on this model cf.
\cite{AD,AkDZ1996,AnDSZ2008,DSZ2009}.
\subsection{Generalization: The mathematical problem}
Motivated by the properties of the physical problem, for the sake
of a more general mathematical setting, in our analysis we
consider the following initial-{} and boundary-value problem of
Schr\"odinger type with variable coefficients and mixed boundary
conditions (Dirichlet-Robin)
\begin{equation}\label{divgen}
\begin{split}
&u_{r}=\mathrm{i\;div}(D\nabla u)+b\nabla u+\mathrm{i}\beta
u+F\;\;\;\;\;\;\;\;
\mbox{in}\;\;\;[r_{\rm min},r_{\rm max}]\times \mathfrak{D},\\
&u=0\;\;\;\;\;\;\;\;\;\;\;\;\;\;\;\;\;\;\;\;\;\;\;\;\;\;\;\;\;\;\;\;\,\;\;\mbox{at}\;\;\;y=0\;\;\;\forall
(r,\theta)\in[r_{\rm min},r_{\rm max}]\times[\theta_{\rm min},\theta_{\rm max}],\\
&\eta^t(D\nabla u) =\mathrm{i}\lambda
u\;\;\;\;\;\;\;\;\;\;\;\;\;\;\;\;\;\;\;\mbox{at}\;\;\;y=1\;\;\;\forall
(r,\theta)\in[r_{\rm min},r_{\rm max}]\times[\theta_{\rm min},\theta_{\rm max}],\\
&u(r_{\rm min},y,\theta)=u_0(y,\theta)
\;\;\;\;\;\;\,\;\;\forall(y,\theta)\in \overline{\mathfrak{D}}.
\end{split}
\end{equation}
Here, $\mathfrak{D}=(0,1)\times(\theta_{\rm min},\theta_{\rm
max})$, $\eta:=(1,0)^t$, while $\beta=\beta(r,y,\theta)$,
$F=F(r,y,\theta)$ and $\lambda=\lambda(r,\theta)$ are complex-valued
functions.

For the rest of this paper, we shall assume that the following
conditions are satisfied:
\begin{equation}\label{con0}
D=D(r,y,\theta)\;\text{is a $2\times 2$ real, symmetric matrix with}\;{\rm det}(D)>0\quad\forall\,r,y,\theta,
\end{equation}
\begin{equation}\label{con1}
b=\Big{(}b_1(r,y,\theta),\;b_2(r,y,\theta)\Big{)}\;\;\mbox{is
real},
\end{equation}
and
\begin{equation}\label{con2}
b_1(r,1,\theta)-2{\rm Re}\lambda(r,\theta)\leq 0\;\;\forall\,
r,\theta.
\end{equation}
\begin{remark}
Since $D\in \mathbb{R}^{2\times 2}$, the condition \eqref{con0}
gives equivalently that  $D$ is either positive or negative
definite for any $r,y,\theta$, which in turn relates to the
ellipticity  of the operator ${\rm div}(D\nabla\cdot)$.
\end{remark}
\begin{remark}
As we shall prove later, the conditions \eqref{con1} and
\eqref{con2} are sufficient for $L^2$-stability, while when
\eqref{con2} holds as equality the problem is $H^1$-stable also,
cf. Theorem \ref{3.3.1} and Remark \ref{h1st}.
\end{remark}
\begin{remark}
The form of the Robin boundary condition, considering only the
first order terms, is related to the elliptic regularity of
elliptic problems with mixed Dirichlet-Robin conditions in two
dimensions proved in Theorem \ref{3.2.10}. The autonomous Section
4 of this paper presents a detailed proof of this argument.
\end{remark}
\begin{remark}
The acoustic problem \eqref{cvpro} is a specific case of the
problem \eqref{divgen} for $u:=v$, $u_0:=v_0$,
$D:=\widehat{D}_{\rm a}$, $b:=(y\frac{s_r}{s},0)$,
$\beta:=\beta_v$, $F:=0$ and $\lambda:=\gamma_{\rm bc}$,
satisfying \eqref{con0}, \eqref{con1}, and \eqref{con2} as
equality, \cite{refthes}.
\end{remark}

\subsection{Main results}
The problem analyzed here is motivated by an important physical
application. Nevertheless, the general mathematical setting
encompasses the very interesting aspect of approximating
numerically a multi-dimensional ibvp of Schr\"odinger-type with
mixed conditions and coefficients depending on the evolutionary
variable.

In this paper, we apply the Galerkin method on the general problem
\eqref{divgen} using piecewise polynomial finite element spaces.
We construct fully discrete Crank--Nicolson-type schemes in $r$
for which we prove stability and optimal rate of accuracy in the
$L^2$-norm. Numerical verification of the optimal rate of convergence
is also presented.

The weak formulation of the problem is presented in Section 2. We
define an appropriate $r$-dependent sesquilinear form which is,
in general, non-Hermitian. As it is common, the rate of accuracy
is investigated by using certain properties of the projection
induced by this form. The projection being $r$-dependent and the
fact that a two-dimensional $r$-dependent Robin boundary
condition appears in \eqref{divgen} make the analysis difficult.
We estimate the projection error and its $r$-derivative in the
$H^1$- and $L^2$-norms (cf. paragraph 2.3, Propositions
\ref{prop1}-\ref{prop4}). The later is accomplished by applying
an Elliptic Regularity Theorem for two-dimensional complex
boundary value problems with mixed Dirichlet and Robin
conditions, proved in Section 5. In the proof of Proposition
\ref{prop4}, where the $r$-derivative of the projection error is
estimated in the $L^2$-norm, we present a very refined argument
when treating the boundary terms.

In Section 3, we write \eqref{divgen} in a weak form and prove
$L^2$-stability, and $H^1$-stability in the case where \eqref{con2} holds as
equality, so that the sesquilinear form is Hermitian. We
then construct a fully discrete Crank-Nicolson scheme in range
$r$ that is shown to be $L^2$-stable. Even though the
evolutionary variable is discretized by a standard Crank-Nicolson
method, the error analysis presented in this section is
non-standard. This is due mainly to the fact that the form and
the projection used are $r$-dependent and calculated at the
mid-points of a uniform range partition. We define properly a
test function split in two terms involving projections applied
on second order derivatives (cf. Remark \ref{tf}), use the
projection estimates of Section 2, and derive an optimal error
estimate in the $L^2$-norm.

A general complex elliptic boundary value problem posed on a
two-dimensional rectangular domain with mixed boundary conditions
is analyzed in Section 4. If Dirichlet or Neumann conditions hold
along the boundary, then in the weak formulation of the boundary
value problem the trace integral terms vanish. A general approach
of proving global regularity, \cite{ref26}, is to prove this
estimate for half-balls, and then by change of variables, stretch
the compact boundary locally and cover it by a finite union of
half-balls. In our case, we analyze a complex elliptic problem
posed on a rectangular domain of $\mathbb{R}^2$. The boundary is
compact and consists of four linear segments along which
Dirichlet and Robin conditions are imposed. We apply directly on
this domain the half-balls technique without change of variables
as the boundary is already stretched locally. Further, we define
appropriate test functions, in order to eliminate the trace terms
from the weak formulation of the problem and prove the regularity
estimate in Theorem \ref{3.2.9}. The result is extended in
Theorem \ref{3.2.10}. Our proof covers a class of Robin
conditions related to the coefficients of the pde of the boundary
value problem, a special case of which is the Abrahamsson-Kreiss
condition of underwater acoustics.

Finally, in Section 5 we report on the results of some numerical
experiments performed with our method, verifying experimentally
the optimal order of convergence.

%
%
\section{An elliptic projection}
\subsection{Preliminaries}
%
Let $\mathfrak{D} = (0,1)\times(\az_1,\az_2)$. For $r$ in
$[r_{\rm min},r_{\rm max}]$ fixed, we define
\begin{equation*}
\partial_{y=y_0}: = \{(r,y_0,\az)\in \mathbb{R}^3: \az\in[\az_1,\az_2]\},\quad
\partial_{\az=\az_0}: = \{(r,y,\az_0)\in \mathbb{R}^3:
y\in[0,1]\},
\end{equation*}
and denote by $H^1(\mathfrak{D})$ the associated usual (complex)
Sobolev space. In order to deal with the Dirichlet boundary condition
we shall make use of the space
\begin{equation*}
\HH = \{u\in H^1(\mathfrak{D}): u|_{\partial\mathfrak{D}_D} = 0\},
\end{equation*}
where $\partial\mathfrak{D}_D:=(\szm) \cup (\ste) \cup (\std)$.
$\HH$ is the space of functions in $H^1(\mathfrak{D})$ which
vanish on $y=0,\az=\az_1,\az=\az_2$. We denote the
$L^2(\mathfrak{D})$ inner product by
$(\cdot,\cdot):L^2(\mathfrak{D})\times
L^2(\mathfrak{D})\rightarrow \mathbb{C}$.
$\|u\|:=\Big{(}\int_{\mathfrak{D}}|u|^2\Big{)}^{\frac{1}{2}}$
denotes the induced $L^2$-norm, while
$\|u\|_1:=\Big{(}\|u\|^2+\|u_y\|^2+\|u_{\az}\|^2\Big{)}^{\frac{1}{2}}$
is the usual $H^1(\mathfrak{D})$ norm.

Let $\partial\mathfrak{D}_R:=\partial_{y=1}$ be the part of
$\partial\mathfrak{D}$ where the Robin boundary condition of
problem \eqref{divgen} is posed, and let
\begin{equation*}
<u,v>:=\int_{\theta_{\rm max}}^{\theta_{\rm
min}}u(1,\theta)\overline{v}(1,\theta)d\theta,
\end{equation*}
denote the inner product on $H^{\frac{1}{2}}(\partial\mathfrak{D}_R)$. In
addition, we shall make use of the norms:
\begin{equation*}
\begin{split}
&|g|_{\frac{1}{2},\partial\mathfrak{D}_R}:=\displaystyle{\inf_{v\in\HH:v|_{\partial\mathfrak{D}_R}=g}}\|v\|_1,\\
&|u|_{-\frac{1}{2},\partial
\mathfrak{D}_R}:=\displaystyle{\sup_{\tilde{v}\in
H^1(\mathfrak{D}),\tilde{v}\neq
0}}\frac{|<u,\tilde{v}>|}{\|\tilde{v}\|_1}.
\end{split}
\end{equation*}

Let $\tau\in{\mathbb N}$ and $S_h$ be a finite dimensional
subspace of $\HH$ consisting of complex-valued functions that are
polynomials of degree less than or equal to $\tau$ in each interval of
a non-uniform partition of $\mathfrak{D}$ with maximum length
$h\in(0,h_{\star}]$. It is well-known, \cite{ref14}, that the
following approximation property holds:
\begin{equation}\label{2.35n}
\begin{split}
\displaystyle{\inf_{\chi\in S_h}}\big\{
\|v-\chi\|+h\|v-\chi\|_1\big\}\leq
C\,h^{s+1}\,\|v\|_{s+1},\quad&\forall\,v\in H^{s+1}(\mathfrak{D}),\\
&\forall\,h\in(0,h_{\star}],\quad s=0,\dots,\tau.
\end{split}
\end{equation}
Also, we assume that the following inverse inequality holds:
\begin{equation}\label{2.36n}
\|\phi\|_1\leq\,C\,h^{-1} \,\|\phi\|\quad\forall\,\phi\in
S_h,\quad\forall\,h\in(0,h_{\star}],
\end{equation}
which is true when, for example, the partition of $\mathfrak{D}$
is quasi-uniform, \cite{ref14}.
%

%
%
\subsection{Definition of a sesquilinear form}
Without loss of generality we assume that $D$ is positive
definite. For any $r$ in $[r_{\rm min},r_{\rm max}]$ we define the
sesquilinear form
$\mathcal{B}(r;\cdot,\cdot):\HH\times\HH\rightarrow \mathbb{C}$
\begin{equation}\label{formdef}
\begin{split}
\mathcal{B}(r;v,w):=&(D(r)\nabla v,\nabla
w)-\mathrm{i}\Big{(}\int_{\theta_{\rm min}}^{\theta_{\rm
max}}\lambda(r,\theta)
v(1,\theta)\overline{w}(1,\theta)d\theta-(b(r)
\nabla v,w)\Big{)}\\
&+\delta(v,w),
\end{split}
\end{equation}
for $\delta$ a sufficiently large positive constant. Obviously, it
holds that
\begin{equation}\label{cont}
|\mathcal{B}(r;v,w)|\leq c\|v\|_1\|w\|_1,
\end{equation}
for any $v,w\in \HH$, uniformly in $r$.

We observe that $(D\nabla v,\nabla v)\geq c_0\|\nabla v\|^2$ for
a constant $c_0>0$ uniformly in $r$ and $v$, since $D$ is real
symmetric and positive definite. Therefore, by the trace inequality
we obtain for $v\in\HH$
\begin{equation*}
\begin{split}
{\rm Re}\mathcal{B}(r;v,v)=&(D(r)\nabla v,\nabla
v)+\delta\|v\|^2\\
&+{\rm Im} \Big{(}\int_{\theta_{\rm min}}^{\theta_{\rm
max}}\lambda(r,\theta)
v(1,\theta)\overline{v}(1,\theta)d\theta-(b(r)\nabla v,v)\Big{)}\\
\geq& c_0\|\nabla v\|^2+\delta\|v\|^2-c^2\|v\|\|v\|_1.
\end{split}
\end{equation*}
Thus, by choosing $\delta$ sufficiently large, it follows that
there exists positive constant $C$ such that
\begin{equation}\label{coer}
\begin{split}
{\rm Re}\mathcal{B}(r;v,v)\geq C\|v\|_1^2,
\end{split}
\end{equation}
uniformly, for any $r$ and any $v\in\HH$.
\begin{remark}
If $D$ is negative definite we may use
\begin{equation*}
\begin{split}
\mathcal{B}_n(r;v,w):=&-\Big{[}(D(r)\nabla v,\nabla
w)-\mathrm{i}\Big{(}\int_{\theta_{\rm min}}^{\theta_{\rm
max}}\lambda(r,\theta)
v(1,\theta)\overline{w}(1,\theta)d\theta-(b(r)\nabla v,w)\Big{)}\Big{]}\\
&+\delta(v,w),
\end{split}
\end{equation*}
with $\delta$ a sufficiently large positive constant. In this case
 \eqref{cont} and \eqref{coer} also hold since now $-D$ is positive
definite.
\end{remark}
\subsection{Projection estimates}
%
%
%
Let $R_h(r):\HH\rightarrow S_h$ be a projection operator defined by
\begin{equation}\label{prodef}
\mathcal{B}(r;R_h(r)v,\phi)=\mathcal{B}(r;v,\phi)\;\;\forall\phi\in
S_h.
\end{equation}
Obviously, since \eqref{cont} and \eqref{coer} hold true, then by
Lax-Milgram Theorem the projection is well defined.

Let us now define the operator
\begin{equation}\label{L*op}
\begin{split}
&\mathcal{L}^*(r)w:=-{\rm div}(D\nabla
w)+\mathrm{i}b\nabla w+\mathrm{i}[b_{1y}+b_{2\theta}]w+\delta w\;\;\mbox{in}\;\;\mathfrak{D},\\
&w=0\;\;\mbox{on}\;\;\partial\mathfrak{D}_D,\\
&\eta^t(D\nabla w)=\mathrm{i}\lambda^*
w\;\;\mbox{on}\;\;\partial\mathfrak{D}_R,
\end{split}
\end{equation}
with $\lambda^*$  a complex-valued function to be chosen appropriately in the
sequel. For $\phi\in\HH$ we get
\begin{equation}\label{wl^*op}
\begin{split}
(\mathcal{L}^*(r)w,\phi) &= (D\nabla w,\nabla\phi)-
\int_{\theta_{\rm min}}^{\theta_{\rm max}}\mathrm{i}\lambda^*
w(1)\overline{\phi}(1)d\theta\\
&\quad+\mathrm{i}(b\nabla
w,\phi)+\mathrm{i}([b_{1y}+b_{2\theta}]w,\phi)+\delta(w,\phi)\\
& = (D\nabla w,\nabla\phi)- \int_{\theta_{\rm min}}^{\theta_{\rm
max}}\mathrm{i}\lambda^*
w(1)\overline{\phi}(1)d\theta\\
&\quad -\mathrm{i}([b_{1y}+b_{2\theta}]w,\phi)
-\mathrm{i}(bw,\nabla\phi) +\mathrm{i}\int_{\theta_{\rm
min}}^{\theta_{\rm
max}}b_1(1)w(1)\overline{\phi}(1)d\theta\\
&\quad +\mathrm{i}([b_{1y}+b_{2\theta}]w,\phi)+\delta(w,\phi)\\
&= (D\nabla w,\nabla\phi)+\mathrm{i}\int_{\theta_{\rm
min}}^{\theta_{\rm max}}[b_1(1)-\lambda^*]
w(1)\overline{\phi}(1)d\theta-\mathrm{i}(bw,\nabla\phi)+\delta(w,\phi).
\end{split}
\end{equation}
Since $D$, $b$ and $\delta$ are real, then for any $\phi$ in $\HH$
it follows that
\begin{equation}\label{wl^*op1}
\begin{split}
\overline{(\mathcal{L}^*(r)w,\phi)} &= (D\nabla \phi,\nabla
w)-\mathrm{i}\int_{\theta_{\rm min}}^{\theta_{\rm
max}}[b_1(1)-\overline{\lambda^*}]
\phi(1)\overline{w}(1)d\theta\\
&\quad +\mathrm{i}(b\nabla \phi,w)+\delta(\phi,w).
\end{split}
\end{equation}
Setting
\begin{equation}\label{defl*}
\lambda^*:=b_1(1)-\overline{\lambda},
\end{equation}
we obtain
\begin{equation}\label{wl^*op2}
\begin{split}
\overline{(\mathcal{L}^*(r)w,\phi)}=\mathcal{B}(r;\phi,w),
\end{split}
\end{equation}
and thus
\begin{equation}\label{wl^*op3}
\begin{split}
(\mathcal{L}^*(r)w,\phi)=\overline{\mathcal{B}(r;\phi,w)},
\end{split}
\end{equation}
for any $\phi\in\HH$. Throughout the rest of this paper, we
consider $\lambda^*$ given by \eqref{defl*}.
\begin{remark}
We observe that in the case of the specific problem
\eqref{cvpro}, $b_1=y\frac{s_r}{s}$, $b_2=0$,
$\lambda=\frac{\mathrm{1}}{2}\Big{[}\frac{s_{r}}{s}+\mathrm{i}\frac{a}{r^2}\Big{(}\frac{s_{\theta}}{s}\Big{)}^2\Big{]}$
and thus
$$\lambda^*=b_1(1)-\overline{\lambda}=\frac{s_r}{s}-
\frac{\mathrm{1}}{2}\Big{[}\frac{s_{r}}{s}-\mathrm{i}\frac{a}{r^2}\Big{(}\frac{s_{\theta}}{s}\Big{)}^2\Big{]}=\lambda.$$
\end{remark}
\begin{proposition}\label{prop1}
There exists a positive constant $c$ such that if $v\in \HH\cap
H^{s}(\mathfrak{D})$ then
\begin{equation}\label{H1e}
\begin{split}
\|R_h(r)v-v\|_1\leq ch^{\tau}\|v\|_{\tau+1},
\end{split}
\end{equation}
and
\begin{equation}\label{L2e}
\begin{split}
\|R_h(r)v-v\|\leq ch^{\tau+1}\|v\|_{\tau+1}.
\end{split}
\end{equation}
\end{proposition}
\begin{proof}
We set $e:=R_h(r)v-v$, use \eqref{coer}, \eqref{cont} and
\eqref{2.35n} to obtain for $\phi\in S_h$
\begin{equation*}
\begin{split}
c\|e\|_1^2\leq {\rm Re}\mathcal{B}(r;e,e)&={\rm
Re}\mathcal{B}(r;e,R_h(r)v-v)={\rm Re}\mathcal{B}(r;e,\phi-v)\\
&\leq c\|e\|_1\displaystyle{\inf_{\phi\in S_h}}\|\phi-v\|_1\leq
c\|e\|_1h^{\tau}\|v\|_{\tau+1},
\end{split}
\end{equation*}
which establishes \eqref{H1e}.

Let now $w$ be the solution of the problem: $\mathcal{L}^*(r)w=e$.
Then by using \eqref{wl^*op3}, the approximation property and
elliptic regularity, proved in Theorem \ref{3.2.10}, we get for
$\phi\in S_h$:
\begin{equation*}
\begin{split}
\|e\|^2=(\mathcal{L}^*(r)w,e)&=\overline{\mathcal{B}(r;e,w)}=\overline{\mathcal{B}(r;e,w-\phi)}\\
&\leq c\|e\|_1\displaystyle{\inf_{\phi\in S_h}}\|w-\phi\|_1\leq
ch^{\tau}\|v\|_{\tau+1}h\|w\|_2\leq
ch^{\tau+1}\|v\|_{\tau+1}\|e\|,
\end{split}
\end{equation*}
which yields \eqref{L2e}.
\end{proof}
\begin{proposition}\label{prop3}
Let $r\in C^1([r_{\rm min},r_{\rm max}],H^{{\rm
r}}(\mathfrak{D}))$. Then it holds that
\begin{equation}\label{H1re}
\left\|\partial_r\left(R_h(r)v(r)-v(r)\right)\right\|_1\leq\,
C\,h^{\tau}\,\left(\,\|v\|_{\tau+1}+\|\partial_r v\|_{\tau+1}\,\right).
\end{equation}
\end{proposition}
\begin{proof}
We set $e:=R_h(r)v(r)-v(r)$. Let $v:[r_{\rm min},r_{\rm
max}]\rightarrow H^{{\rm r}}(\mathfrak{D})$ and
$e(r)=R_h(r)v(r)-v(r)$ for $r\in[r_{\rm min},r_{\rm max}]$. Then,
we have
\begin{equation*}
\mathcal{B}(r;e(r),\phi)=0\;\;\forall\phi\in S_h.
\end{equation*}
Differentiating the above relation with respect to $r$ we obtain
\begin{equation*}
\mathcal{B}(r;\dot{e}(r),\phi)+\dot{\mathcal{B}}(r;e(r),\phi)=0\;\;\forall\phi\in
S_h.
\end{equation*}
Now,  for $\phi\in S_h$ we have
\begin{equation*}
\begin{split}
c\|\dot{e}(r)\|_1^2&\leq
\mathcal{B}(r;\dot{e}(r),\dot{e}(r))=\mathcal{B}(r;\dot{e}(r),\dot{e}(r)+\phi)
-\mathcal{B}(r;\dot{e}(r),\phi)\\
&=\mathcal{B}(r;\dot{e}(r),\dot{e}(r)+\phi)
+\mathcal{B}(r;e(r),\phi)\\
&\leq
c\Big{[}\|\dot{e}(r)\|_1\|\dot{e}(r)+\phi\|_1+\|e(r)\|_1\|\phi\|_1\Big{]}\\
&\leq
c\Big{[}\|\dot{e}(r)\|_1\|\dot{e}(r)+\phi\|_1+\|e(r)\|_1(\|\dot{e}(r)+\phi\|_1+\|\dot{e}(r)\|_1)\Big{]}\\
&=c\Big{[}(\|\dot{e}(r)\|_1+\|e(r)\|_1)\|\dot{e}(r)+\phi\|_1+\|e(r)\|_1\|\dot{e}(r)\|_1\Big{]}\\
&\leq
c\Big{[}(\|\dot{e}(r)\|_1+\|e(r)\|_1)\displaystyle{\inf_{\phi\in
S_h}} \|\partial_r(R_hv)(r)-\partial_r
v+\phi\|_1+\|e(r)\|_1\|\dot{e}(r)\|_1\Big{]}\\
&\leq
c\Big{[}(\|\dot{e}(r)\|_1+\|e(r)\|_1)\displaystyle{\inf_{\phi\in
S_h}} \|\partial_r
v-\phi\|_1+\|e(r)\|_1\|\dot{e}(r)\|_1\Big{]}\\
&= c\|\dot{e}(r)\|_1\Big{[}\|e(r)\|_1+\displaystyle{\inf_{\phi\in
S_h}} \|\partial_r
v-\phi\|_1\Big{]}+c\|e(r)\|_1\displaystyle{\inf_{\phi\in S_h}}
\|\partial_r v-\phi\|_1.
\end{split}
\end{equation*}
The claim of the proposition follows by using the
approximation property \ref{2.35n} with $s=\tau+1$.
\end{proof}
Using a technique introduced in \cite{Dupont1973}, we are able
to show the following optimal order approximation result for
the time-derivative of the elliptic projection.
\begin{proposition}\label{prop4}
There exists a positive constant $c$ such that
\begin{equation}\label{L2re}
\left\|\partial_r\left(R_h(r)v(r)-v(r)\right)\right\|\leq
C\,h^{\tau+1}\,\left(\,\|v\|_{\tau+1}+\|\partial_r v\|_{\tau+1}\,\right).
\end{equation}
\end{proposition}
%
%
\begin{proof}
We set $e:=R_h(r)v(r)-v(r)$. Let $w$ be the solution of the
problem: $\mathcal{L}^*w=\dot{e}$. For $\chi\in S_h$ we have
\begin{equation*}
\begin{split}
\|\dot{e}(r)\|^2&=(\mathcal{L}^*w,\dot{e}(r))=\overline{\mathcal{B}(r;\dot{e}(r),w)}\\
&={\rm
Re}\Big{[}\overline{\mathcal{B}(r;\dot{e}(r),w-\chi)}-\overline{\dot{\mathcal{B}}(r;e,\chi)}\Big{]}\\
&\leq
c\Big{[}\|\dot{e}(r)\|_1\|w-\chi\|_1+\|e(r)\|_1\|w-\chi\|_1\Big{]}-{\rm
Re}\Big{[}\dot{\mathcal{B}}(r;e,w)\Big{]}\\
&\leq
c\Big{(}\|\dot{e}(r)\|_1+\|e(r)\|\Big{)}\displaystyle{\inf_{\chi\in
S_h}}\|w-\chi\|_1-{\rm
Re}\Big{[}\dot{\mathcal{B}}(r;e,w)\Big{]}\\
&\leq ch^{\rm r} \Big{(}\|v\|_{\rm r+1}+\|\partial_r v \|_{\rm
r+1} \Big{)}h\|w\|_2-{\rm
Re}\Big{[}\dot{\mathcal{B}}(r;e,w)\Big{]}.
\end{split}
\end{equation*}
For  convenience we set $I:={\rm
Re}\Big{[}\dot{\mathcal{B}}(r;e,w)\Big{]}$. First,  observe that
\begin{equation*}
\dot{\mathcal{B}}(r;e,w)=(\partial_r D\nabla e,\nabla
w)-\mathrm{i}\Big{(}\int_{\theta_{\rm min}}^{\theta_{\rm
max}}\partial_r\lambda(r,\theta)
e(r,1,\theta)\overline{w}(r,1,\theta)d\theta-(\partial_rb \nabla
e,w)\Big{)},
\end{equation*}
so that
\begin{equation*}
I={\rm Re}\Big{[}(\partial_r D\nabla e,\nabla
w)-\mathrm{i}\Big{(}\int_{\theta_{\rm min}}^{\theta_{\rm
max}}\partial_r\lambda(r,\theta)
e(r,1,\theta)\overline{w}(r,1,\theta)d\theta-(\partial_rb \nabla
e,w)\Big{)}\Big{]}.
\end{equation*}
By the definition of  the inner product $<u,v>$
we have
\begin{equation*}
\begin{split}
I_1:={\rm Re}\Big{[}\mathrm{i}(\partial_rb \nabla e,w)\Big{]}={\rm
Re}\Big{[}&-\mathrm{i}(\partial_{r}[b_{1y}+b_{2\theta}]e,w)-\mathrm{i}(\partial_rb
e,\nabla w)\\
& +\mathrm{i}\int_{\theta_{\rm min}}^{\theta_{\rm
max}}(\partial_rb_1)(1)e(1)\overline{w}(1)d\theta\Big{]}\\
&\leq c\|e\|\|w\|_1+{\rm Re}\Big{[}\mathrm{i}<\partial_rb_1
e,w>\Big{]}.
\end{split}
\end{equation*}
We set $I_2:={\rm Re}\Big{[}-\mathrm{i}\int_{\theta_{\rm
min}}^{\theta_{\rm max}}\partial_r\lambda
e\overline{w}d\theta\Big{]}$. Using the estimates above we obtain
\begin{equation*}
\begin{split}
I_1+I_2&\leq {\rm
Re}\Big{[}-\mathrm{i}<[\partial_r\lambda-\partial_rb_1]e,w>\Big{]}+
c\|e\|\|w\|_1\\
&\leq c|e|_{-\frac{1}{2},\partial
\mathfrak{D}_R}\|w\|_1+c\|e\|\|w\|_1\leq c
\Big{[}\|e\|+|e|_{-\frac{1}{2},\partial
\mathfrak{D}_R}\Big{]}\|w\|_1.
\end{split}
\end{equation*}
In addition,
\begin{equation*}
\begin{split}
I_3&:={\rm Re}\Big{[}(\partial_r D\nabla e,\nabla w)\Big{]}={\rm
Re}\Big{[}\int_{\theta_{\rm min}}^{\theta_{\rm
max}}\eta(\partial_rD\nabla w\overline{e})d\theta-({\rm
div}(\partial_rD\nabla w),e)\Big{]}\\
&\leq c\|w\|_2\|e\|+{\rm
Re}\Big{[}<\eta(\partial_rD\nabla w),e>\Big{]}\\
&\leq c\|w\|_2\|e\|+|e|_{-\frac{1}{2},\partial
\mathfrak{D}_R}\|w\|_2\leq
c\Big{[}\|e\|+|e|_{-\frac{1}{2},\partial
\mathfrak{D}_R}\Big{]}\|w\|_2,
\end{split}
\end{equation*}
so that
\begin{equation}\label{Iest}
I\leq c\Big{[}\|e\|+|e|_{-\frac{1}{2},\partial
\mathfrak{D}_R}\Big{]}\|w\|_2.
\end{equation}

Now, for $g\in H^{\frac{1}{2}}(\partial \mathfrak{D}_R)$ we
consider the elliptic problem
\begin{equation*}
\begin{split}
&-{\rm div}(D\nabla
z)+\mathrm{i}b\nabla z+\mathrm{i}[b_{1y}+b_{2\theta}]z+\delta z=0\;\;\mbox{in}\;\;\mathfrak{D},\\
&z=0\;\;\mbox{on}\;\;\partial\mathfrak{D}_D,\\
&\eta^t(D\nabla z)=\mathrm{i}\lambda^*
z+g\;\;\mbox{on}\;\;\partial\mathfrak{D}_R.
\end{split}
\end{equation*}
Then we have $0=\overline{\mathcal{B}(r;e(r),z)}-<g,e>$ and thus,
\begin{equation*}
\begin{split}
<g,e>=\overline{\mathcal{B}(r;e(r),z)}=\overline{\mathcal{B}(r;e(r),z-\phi)}\;\;\forall\phi\in
S_h.
\end{split}
\end{equation*}
It follows then that
\begin{equation*}
\begin{split}
|<e,g>|\leq c\|e\|_1\displaystyle{\inf_{\phi\in S_h}}\|z-\phi\|_1,
\end{split}
\end{equation*}
and therefore,
$$|<e,g>|\leq ch^{\tau}\|v\|_{\tau+1}h\|z\|_2.$$
The elliptic regularity result (cf. Theorem \ref{3.2.10} and
Remark \ref{gterm2}) for the solution $z$ of the elliptic problem
above, reads
$$\|z\|_2\leq c|g|_{\frac{1}{2},\partial\mathfrak{D}_R}.$$
Thus we have
\begin{equation*}
|e|_{-\frac{1}{2},\partial
\mathfrak{D}_R}:=\displaystyle{\sup_{\tilde{v}\in
H^1(\mathfrak{D}),\tilde{v}\neq
0}}\frac{|<e,\tilde{v}>|}{\|\tilde{v}\|_1}\leq
ch^{\tau+1}\|v\|_{\tau+1},
\end{equation*}
and subsequently, using the elliptic regularity of $w$, cf. again
Theorem \ref{3.2.10}, we arrive at
\begin{equation*}
\begin{split}
\|\dot{e}(r)\|^2&\leq
ch^{\tau+1}\Big{(}\|v\|_{\tau+1}+\|\partial_r
v\|_{\tau+1}\Big{)}\|w\|_2+c\Big{[}\|e\|+|e|_{-1/2,\partial
\mathfrak{D}_R}\Big{]}\|w\|_2\\
&\leq c\|w\|_2h^{\tau+1}\Big{(}\|v\|_{\tau+1}+\|\partial_r
v\|_{\tau+1}\Big{)}\leq
c\|\dot{e}(r)\|h^{\tau+1}\Big{(}\|v\|_{\tau+1}+\|\partial_r
v\|_{\tau+1}\Big{)},
\end{split}
\end{equation*}
which completes the proof  of the proposition.
\end{proof}
\section{A Crank--Nikolson-type fully discrete scheme}
\subsection{Weak Formulation}
Let $\phi\in \HH$. Multiplying the partial differential equation
of \eqref{divgen} by $\bar{\phi}$ and integrating by parts
we have
\begin{equation}\label{3.70}
\begin{split}
&\Big{(}u_r(r),\phi\Big{)}= \mathrm{i}\Big{[}-\Big{(}D(r)\nabla
u(r),\nabla\phi\Big{)}+\int_{\partial\mathfrak{D}}\eta^t
D(r)\nabla u(r)\overline{\phi}ds\Big{]}
+\Big{(}b(r)\nabla u(r),\phi\Big{)}\\
&+\mathrm{i}\Big{(}\beta(r) u(r),\phi\Big{)}+(F(r),\phi)\\
&=-\mathrm{i}\Big{[}\Big{(}D(r)\nabla
u(r),\nabla\phi\Big{)}-\mathrm{i}\Big{\{}\int_{\theta_{\rm
min}}^{\theta_{\rm max}}\lambda(r,\theta)
u(r,1,\theta)\overline{\phi}(1,\theta)d\theta-\Big{(}b(r)\nabla u(r),\phi\Big{)}\Big{\}}\Big{]}\\
&+\mathrm{i}\Big{(}\beta(r) u(r),\phi\Big{)}+(F(r),\phi)\\
&=-\mathrm{i}\mathcal{B}\Big{(}r;u(r),\phi\Big{)}+\mathrm{i}\Big{(}(\beta(r)+\delta)
u(r),\phi\Big{)}+(F(r),\phi),
\end{split}
\end{equation}
for any $\phi\in\HH$. In the following theorem we prove that
\eqref{3.70} defines  $u$ in $\HH$ uniquely.
%
\begin{theorem}\label{3.3.1}
The weak problem \eqref{3.70} has at most one solution in $\HH$.
\end{theorem}
\begin{proof}
Let $u\in\HH$ be a solution of \eqref{3.70}. We set $\phi=u$ in
\eqref{3.70}, integrate by parts, use the facts that $D$ is a
real, symmetric matrix, that $b$ is real, and take real parts.
More specifically, we obtain first
\begin{equation}\label{3.3.1a}
\begin{split}
(u_r,u)= &-\mathrm{i}\Big{[}(D\nabla u,\nabla u
)-\mathrm{i}\Big{(}\int_{\theta_{\rm min}}^{\theta_{\rm
max}}\lambda(r,\theta)
u(r,1,\theta)\overline{u}(1,\theta)d\theta-(b
\nabla u,u)\Big{)}\Big{]}\\
&+\mathrm{i}(\beta u,u)+(F,u).
\end{split}
\end{equation}
Observe that $${\rm Re}(b\nabla u,u)=-\frac{1}{2}(b_{1y}
u,u)-\frac{1}{2}(b_{2\theta} u,u)+\frac{1}{2}\int_{\theta_{\rm
min}}^{\theta_{\rm
max}}b_1(r,1,\theta)|u(r,1,\theta)|^2d\theta,$$ since $b$ is real
and $u=0$ at $y=0$, $\theta=\theta_{\rm min},\theta_{\rm max}$.
Since $D$ is real, then using this observation in \eqref{3.3.1a}
we arrive at
\begin{equation*}
\begin{split}
\frac{1}{2}\frac{d}{dr}\|u\|^2=&\int_{\theta_{\rm
min}}^{\theta_{\rm
max}}[-{\rm Re}\lambda(r,\theta)+\frac{1}{2}b_1(r,1,\theta)]|u(r,1,\theta)|^2d\theta\\
&-\frac{1}{2}([b_{1y}+b_{2\theta}]u,u)-({\rm Im}(\beta) u,u)+{\rm
Re}(F,u).
\end{split}
\end{equation*}
Using the condition \eqref{con2} and Gr\"onwall's inequality we
obtain the stability estimate
\begin{equation}\label{3.72}
\left.\begin{array}{l} \|u\|\leq c\|u_0\|+c\int_{r_{\rm
min}}^{r_{\rm max}}\|F\|\,dr.
\end{array}\right.
\end{equation}
Uniqueness of the solution $u$ follows readily from the estimate
above.
\end{proof}
\begin{remark}\label{h1st}
If \eqref{con2} holds as equality then the sesquilinear form is
Hermitian. Therefore, if $F=0$,  using \eqref{3.70}, setting
$\phi=u_r-{\rm i}(\beta+\delta)u$ and taking imaginary parts we
obtain
\begin{equation*}
\begin{split}
\frac{1}{2}\frac{d}{dt}{\rm Re}\mathcal{B}(r;u(r),u(r))\leq c_1
\|u\|_1^2\leq c{\rm Re}\mathcal{B}(r;u(r),u(r)),
\end{split}
\end{equation*}
so that $c\|u\|_1^2\leq {\rm Re}\mathcal{B}(r;u(r),u(r))\leq
c\|u_0\|_1^2$, i.e.  we also obtain an $H^1$ stability estimate.
\end{remark}
\begin{remark}
Note that for the specific case of problem \eqref{cvpro}, if
$\beta_\psi$ is real, we have $F=0$, ${\rm Im}(\beta):={\rm
Im}(\beta_v)=-\frac{s_r}{2s}$, $b_1=y\frac{s_r}{s}$, $b_2=0$ and
\eqref{con2} holds as equality, therefore (cf. the proof of the
previous theorem) we obtain the conservation property
$$\|v\|=\|v_0\|$$ for any $r$, while
the problem is also $H^1$--stable.
\end{remark}
\subsection{The numerical scheme}
For $N > 1$ integer, we consider a uniform partition in range
$r_{\rm min} =r^0<r^1<\cdots<r^N=r_{\rm max}$, $0\leq n\leq N$,
$k:=r^{n+1}-r^n=\frac{1}{N}$ for any $n\leq N-1$, and set
$r^{n+\frac{1}{2}}:=\frac{r^n+r^{n+1}}{2}$. If $u$ is the
solution of the continuous problem \eqref{divgen},
 we approximate $u(r^{n+1})$ by $U^{n+1}\in S_h$ as follows: for $U^n$
known we seek $U^{n+1}\in S_h$ such that
\begin{equation}\label{3.101}
\begin{split}\Big{(}\frac{U^{n+1}-U^n}{k},\phi\Big{)}=&-{\rm
i}\mathcal{B}\Big{(}r^{\nn};\frac{U^{n+1}+U^n}{2},\phi\Big{)}+{\rm
i}\Big{(}(\beta(\rnn)+\delta)\frac{U^{n+1}+U^n}{2},\phi\Big{)}\\
&+(F(r^{\nn}),\phi),
\end{split}
\end{equation}
for any $\phi\in S_h$, and any $0\leq n\leq N-1$. In order to
obtain an optimal order approximation we shall take
$U^0:=R_h(r^0)u_0\in S_h$.
\begin{remark}\label{rem1}
Let $w\in \HH$. The need of the condition \eqref{con2} appears
once again (recall that \eqref{con2} was used for the $L^2$
stability of the continuous problem). More specifically, since
$w\in\HH$ then by \eqref{con2} the following inequality holds
\begin{equation}\label{*3}
\left.\begin{array}{l} {\rm Re}\Big{\{}-{\rm
i}\mathcal{B}\Big{(}r;w,w\Big{)}\Big{\}}\leq c\|w\|^2.
\end{array}\right.
\end{equation}
\end{remark}
\begin{theorem}\label{3.3.6}
The fully discrete scheme (\ref{3.101}) is $L^2$-stable.
\end{theorem}
\begin{proof}
In (\ref{3.101}) we set $\phi=U^{n+1}+U^n\in S_h\subset \HH$,
take real parts and use the estimate of Remark \ref{rem1} to
arrive at
\begin{equation*}
\left.\begin{array}{l}
 (1-ck)\|U^{n+1}\|\leq (1+ck)\|U^{n}\|+ck\|F(\rnn)\|.
\end{array}\right.
\end{equation*}
Choosing $k$ sufficiently small we get a
stability result for the scheme (\ref{3.101})
\begin{equation*}
\left.\begin{array}{l}
 \|U^n\|\leq c\|U^0\|+c\;\displaystyle{\max_{n\leq N}}\|F\|,
\end{array}\right.
\end{equation*}
for any $1\leq n\leq N$. Consequently, uniqueness of solution
in $S_h$ is also established.
\end{proof}
%
\subsection{Error estimates for the fully discrete scheme}
\subsubsection{Preliminaries}
We define in $S_h\subset\HH$ the quantities
\begin{equation*}
\begin{split}
&\theta_1^n:=U^n-R_h(\rnn)u(r^n)+\frac{k^2}{8}R_h(\rnn)u_{rr}(\rnn),\;\;0\leq n \leq N-1,\\
&\theta_2^{n+1}:=U^{n+1}-R_h(\rnn)u(r^{n+1})+\frac{k^2}{8}R_h(\rnn)u_{rr}(\rnn),\;\;0\leq
n \leq N-1,
\end{split}
\end{equation*}
where $U^n$ is the solution of the fully discrete scheme
(\ref{3.101}).
\begin{remark}\label{tf}
The main idea is to mimic the continuous problem. In the fully
discrete scheme, we set $\phi:=\theta_2^{n+1}+\theta_1^n$ as test
function. The choice of $\theta_1^n$, $\theta_2^{n+1}$ is not
standard and is made in order to treat efficiently the
$r$-dependent sesquilinear form at the midpoints of the partition
and since the projection is range-dependent. Therefore, in
$\phi$, the projections are computed in $\rnn$. The introduction
of the specific additive term
\begin{equation}\label{adpr}
\frac{k^2}{8}R_h(\rnn)u_{rr}(\rnn)
\end{equation}
 is motivated by the
approximation
\begin{equation*}\label{taylor}
\begin{split}
\frac{u(r^n)+u(r^{n+1})}{2}-u(r^{n+\frac{1}{2}})=\frac{k^2}{8}u_{rr}(\rnn)+\mathcal{O}(k^4),
\end{split}
\end{equation*}
used in \eqref{n1**}. The residual being of order
$\mathcal{O}(k^4)$ permits us to apply then in \eqref{n1a} the
inverse inequality without loss of optimality in space, and avoid
thus any integration by parts (denote that in this case
suboptimal trace integral terms would appear, as the problem is
posed in $S_h\subset\HH\neq H_0^1(\mathfrak{D})$). Furthermore,
the term \eqref{adpr} is related to the approximations
\begin{equation*}
\begin{split}
&u(r^{n+1})-\frac{k^2}{8}u_{rr}(\rnn)=u(\rnn)+\frac{k}{2}u_r(\rnn)+\mathcal{O}(k^3),\\
&u(r^{n})-\frac{k^2}{8}u_{rr}(\rnnm)=u(\rnnm)+\frac{k}{2}u_r(\rnnm)+\mathcal{O}(k^3)
\end{split}
\end{equation*}
used in the proof of Lemma \ref{elem2} when treating the
$r$-derivative of the projection error.
\end{remark}

We notice that
\begin{equation*}
\begin{split}
\frac{U^{n+1}-U^n}{k}=&\frac{\theta_2^{n+1}-\theta_1^n}{k}+\frac{R_h(\rnn)u(r^{n+1})-R_h(\rnn)u(r^n)}{k},\\
\frac{U^{n+1}+U^n}{2}=&\frac{\theta_2^{n+1}+\theta_1^n}{2}+\frac{R_h(\rnn)u(r^{n+1})+R_h(\rnn)u(r^n)}{2}\\
&-\frac{k^2}{8}R_h(\rnn)u_{rr}(\rnn).
\end{split}
\end{equation*}
Replacing these identities in the fully discrete scheme we obtain
\begin{equation}\label{3.1033}
 \begin{split}
&\Big{(}\frac{\theta_2^{n+1}-\theta_1^n}{k},\phi\Big{)}=-\Big{(}\frac{R_h(\rnn)u(r^{n+1})-R_h(\rnn)u(r^n)}{k},\phi\Big{)}
\\&-{\rm
i}\mathcal{B}\Big{(}\rnn;\frac{\theta_2^{n+1}+\theta_1^n}{2},\phi\Big{)}\\
&-{\rm
i}\mathcal{B}\Big{(}\rnn;\frac{R_h(\rnn)u(r^{n+1})+R_h(\rnn)u(r^n)}{2},\phi\Big{)}\\
&+{\rm
i}\mathcal{B}\Big{(}\rnn;\frac{k^2}{8}R_h(\rnn)u_{rr}(\rnn),\phi\Big{)}\\
&+{\rm
i}\Big{(}(\beta(\rnn)+\delta)\frac{\theta_2^{n+1}+\theta_1^n}{2},\phi\Big{)}\\
&+{\rm
i}\Big{(}(\beta(\rnn)+\delta)\frac{R_h(\rnn)u(r^{n+1})+R_h(\rnn)u(r^n)}{2},\phi\Big{)}\\
&-{\rm
i}\frac{k^2}{8}\Big{(}(\beta(\rnn)+\delta)R_h(\rnn)u_{rr}(\rnn),\phi\Big{)}
+(F(\rnn),\phi).
\end{split}
\end{equation}
From the continuous problem we have that
\begin{equation}\label{3.103*}
\begin{split}
(\partial_ru(\rnn),\phi)=&-{\rm
i}\mathcal{B}\Big{(}\rnn;u(\rnn),\phi\Big{)}\\
&+{\rm i}
\Big{(}(\beta(\rnn)+\delta)u(\rnn),\phi\Big{)}+(F(\rnn),\phi).
\end{split}
\end{equation}
We now solve \eqref{3.103*} for $(F(\rnn),\phi)$, replace in
(\ref{3.1033}), and use the definition of the elliptic projection
$R_h$ to arrive at
\begin{equation}\label{3.1034}
 \begin{split}
&\Big{(}\frac{\theta_2^{n+1}-\theta_1^n}{k},\phi\Big{)}=-\Big{(}\frac{R_h(\rnn)u(r^{n+1})-R_h(\rnn)u(r^n)}{k}
-u_r(\rnn),\phi\Big{)}\\
&-{\rm
i}\mathcal{B}\Big{(}\rnn;\frac{\theta_2^{n+1}+\theta_1^n}{2},\phi\Big{)}
-{\rm
i}\mathcal{B}\Big{(}\rnn;\frac{u(r^{n+1})+u(r^n)}{2}-u(\rnn),\phi\Big{)}\\
&+{\rm
i}\mathcal{B}\Big{(}\rnn;\frac{k^2}{8}R_h(\rnn)u_{rr}(\rnn),\phi\Big{)}\\
&+{\rm
i}\Big{(}(\beta(\rnn)+\delta)\frac{\theta_2^{n+1}+\theta_1^n}{2},\phi\Big{)}\\
&+{\rm
i}\Big{(}(\beta(\rnn)+\delta)\Big{[}\frac{R_h(\rnn)u(r^{n+1})+R_h(\rnn)u(r^n)}{2}-u(\rnn)\Big{]},\phi\Big{)}\\
&-{\rm
i}\frac{k^2}{8}\Big{(}(\beta(\rnn)+\delta)R_h(\rnn)u_{rr}(\rnn),\phi\Big{)}.
\end{split}
\end{equation}
By Taylor's formula the following identity holds for $r_1,r_2\in
(r^n,r^{n+1})$
\begin{equation*}\label{taylor}
\begin{split}
\frac{u(r^n)+u(r^{n+1})}{2}-u(r^{n+\frac{1}{2}})=\frac{k^2}{8}u_{rr}(\rnn)+\frac{k^4}{2\cdot
16\cdot 4!}[u_{rrrr}(r_1)+u_{rrrr}(r_2)].
\end{split}
\end{equation*}
Using the above in \eqref{3.1034} we obtain
\begin{equation}\label{n1**}
\begin{split}
&-{\rm
i}\mathcal{B}\Big{(}\rnn;\frac{u(r^{n+1})+u(r^n)}{2}-u(\rnn),\phi\Big{)}+{\rm
i}\mathcal{B}\Big{(}\rnn;\frac{k^2}{8}R_h(\rnn)u_{rr}(\rnn),\phi\Big{)}\\
&=-{\rm
i}\mathcal{B}\Big{(}\rnn;-\frac{k^2}{8}\Big{[}R_h(\rnn)u_{rr}(\rnn)-u_{rr}(\rnn)\Big{]},\phi\Big{)}\\
&-{\rm i}\mathcal{B}\Big{(}\rnn;\frac{k^4}{2\cdot 16\cdot
4!}[u_{rrrr}(r_1)+u_{rrrr}(r_2)],\phi\Big{)}\\
&=0-{\rm i}\mathcal{B}\Big{(}\rnn;\frac{k^4}{2\cdot 16\cdot
4!}[u_{rrrr}(r_1)+u_{rrrr}(r_2)],\phi\Big{)}.
\end{split}
\end{equation}
Therefore, applying an inverse inequality we obtain
\begin{equation}\label{n1a}
\begin{split}
&{\rm Re}\Big{[}-{\rm
i}\mathcal{B}\Big{(}\rnn;\frac{u(r^{n+1})+u(r^n)}{2}-u(\rnn)-\frac{k^2}{8}R_h(\rnn)u_{rr}(\rnn),\phi\Big{)}\Big{]}\\
&\leq ck^4\displaystyle{\max_{r}}\|u_{rrrr}\|_1\|\phi\|_1\leq
ck^4h^{-1}\displaystyle{\max_{r}}\|u_{rrrr}\|_1\|\phi\|.
\end{split}
\end{equation}

In addition, the Taylor formula gives
\begin{equation*}
\begin{split}
&{\rm
i}\Big{(}(\beta(\rnn)+\delta)\Big{[}\frac{R_h(\rnn)u(r^{n+1})+R_h(\rnn)u(r^n)}{2}-u(\rnn)\Big{]},\phi\Big{)}\\
&={\rm
i}\Big{(}(\beta(\rnn)+\delta)\Big{[}R_h(\rnn)u(\rnn)-u(\rnn)\Big{]},\phi\Big{)}\\
&+ {\rm
i}\Big{(}(\beta(\rnn)+\delta)R_h(\rnn)\Big{[}\frac{k^2}{8}u_{rr}(\rnn)+\frac{k^4}{2\cdot
16\cdot 4!}[u_{rrrr}(r_1)+u_{rrrr}(r_2)]\Big{]},\phi\Big{)}\\
&={\rm
i}\Big{(}(\beta(\rnn)+\delta)\Big{[}R_h(\rnn)u(\rnn)-u(\rnn)\Big{]},\phi\Big{)}\\
&+ {\rm
i}\Big{(}(\beta(\rnn)+\delta)(R_h(\rnn)-I)\Big{[}\frac{k^2}{8}u_{rr}(\rnn)+\frac{k^4}{2\cdot
16\cdot 4!}[u_{rrrr}(r_1)+u_{rrrr}(r_2)]\Big{]},\phi\Big{)}\\
&+{\rm
i}\Big{(}(\beta(\rnn)+\delta)\Big{[}\frac{k^2}{8}u_{rr}(\rnn)+\frac{k^4}{2\cdot
16\cdot 4!}[u_{rrrr}(r_1)+u_{rrrr}(r_2)]\Big{]},\phi\Big{)}.
\end{split}
\end{equation*}
Thus, we obtain
\begin{equation}\label{n2}
\begin{split}
&{\rm Re}\Big{[}{\rm
i}\Big{(}(\beta(\rnn)+\delta)\Big{[}\frac{R_h(\rnn)u(r^{n+1})+R_h(\rnn)u(r^n)}{2}-u(\rnn)\Big{]},\phi\Big{)}\Big{]}\\
&\leq c\Big{\{}h^{\tau+1}+k^2\Big{\}}\|\phi\|.
\end{split}
\end{equation}

Also Taylor gives for $r_3,r_4\in (r^n,r^{n+1})$
\begin{equation*}
\frac{u(r^{n+1})-u(r^n)}{k}=\frac{k^2}{8\cdot
3!}[u_{rrr}(r_3)+u_{rrr}(r_4)]+u_r(\rnn),
\end{equation*}
therefore,
\begin{equation*}
\begin{split}
&-\Big{(}\frac{R_h(\rnn)u(r^{n+1})-R_h(\rnn)u(r^n)}{k}
-u_r(\rnn),\phi\Big{)}\\
&=-\Big{(}[R_h(\rnn)-I]u_r(r^{n+1}),\phi\Big{)} -\frac{k^2}{8\cdot
3!}\Big{(}[R_h(\rnn)-I][u_{rrr}(r_3)+u_{rrr}(r_4)],\phi\Big{)}\\
&-\frac{k^2}{8\cdot
3!}\Big{(}u_{rrr}(r_3)+u_{rrr}(r_4),\phi\Big{)}.
\end{split}
\end{equation*}
The above yields
\begin{equation}\label{n3}
\begin{split}
&{\rm
Re}\Big{[}-\Big{(}\frac{R_h(\rnn)u(r^{n+1})-R_h(\rnn)u(r^n)}{k}
-u_r(\rnn),\phi\Big{)}\Big{]}\leq
c\Big{\{}h^{\tau+1}+k^2\Big{\}}\|\phi\|.
\end{split}
\end{equation}

Let us now assume that $k<1$ and $k\leq ch^{\frac{1}{2}}$. In
\eqref{3.1034}, we take real parts and use \eqref{n1a},
\eqref{n2}, and \eqref{n3} to obtain
\begin{equation}\label{n4}
  \begin{split}
 &{\rm Re}
\Big{[}\Big{(}\frac{\theta_2^{n+1}-\theta_1^n}{k},\phi\Big{)}\Big{]}\leq
c\Big{\{}h^{\tau+1}+k^2\Big{\}}\|\phi\|\\
&+{\rm Re}\Big{[}-{\rm
i}\mathcal{B}\Big{(}\rnn;\frac{\theta_2^{n+1}+\theta_1^n}{2},\phi\Big{)}+{\rm
i}\Big{(}(\beta(\rnn)+\delta)\frac{\theta_2^{n+1}+\theta_1^n}{2},\phi\Big{)}\Big{]}.
\end{split}
\end{equation}
In the above estimate, we set $\phi:=\theta_2^{n+1}+\theta_1^n\in
S_h\subset\HH$, and use the estimate of  Remark \ref{rem1} to
obtain for $k$ sufficiently small
\begin{equation}\label{n5}
\|\theta_2^{n+1}\|\leq
\Big{(}\frac{1+ck}{1-ck}\Big{)}\|\theta_1^n\|+\mathcal{A},\;\;0\leq
n\leq N-1,
\end{equation}
where $\mathcal{A}\leq\frac{ck(h^{\tau+1}+k^2)}{1-ck}$.

Let us now define
\begin{gather}
\theta^n:=U^n-R_h(r^n)u(r^n),\;\;0\leq n\leq N.
\end{gather}
and
\begin{equation}\label{n6}
\begin{split}
&B_2^{(n+1)}:=-R_h(r^{n+1})u(r^{n+1})+R_h(\rnn)u(r^{n+1})-\frac{k^2}{8}R_h(\rnn)u_{rr}(\rnn),\;\;0\leq n\leq N-1,\\
&B_1^{(n)}:=-R_h(r^{n})u(r^{n})+R_h(\rnn)u(r^{n})-\frac{k^2}{8}R_h(\rnn)u_{rr}(\rnn),\;\;0\leq
n\leq N-1.
\end{split}
\end{equation}
So, we obtain
\begin{equation}\label{n7}
\begin{split}
&\theta_2^{n+1}=\theta^{n+1}-B_2^{(n+1)},\;\;0\leq n\leq N-1,\\
&\theta_1^n=\theta^{n}-B_1^{(n)},\;\;0\leq n\leq N-1.
\end{split}
\end{equation}
We replace in \eqref{n5} so that for any $1\leq n\leq N-1$ we
arrive at
\begin{equation}\label{n8}
\begin{split}
\|\theta^{n+1}-B_2^{(n+1)}\|\leq&
\Big{(}\frac{1+ck}{1-ck}\Big{)}\|\theta^{n}-B_1^{(n)}\|
+\mathcal{A}\\
\leq&
\Big{(}\frac{1+ck}{1-ck}\Big{)}\|\theta^{n}-B_2^{(n)}\|+\mathcal{A}\\
&+\Big{(}\frac{1+ck}{1-ck}\Big{)}\|B_2^{(n)}-B_2^{(n+1)}\|+
\Big{(}\frac{1+ck}{1-ck}\Big{)}\|B_2^{(n+1)}-B_1^{(n)}\|.
\end{split}
\end{equation}
\subsubsection{The estimates}
We prove first the following lemmas.
\begin{lemma}\label{elem1}For any $0\leq n\leq N-1$ it holds that
\begin{equation}\label{n9}
\|B_2^{(n+1)}-B_1^{(n)}\|\leq ckh^{\tau+1}.
\end{equation}
\end{lemma}
\begin{proof}
By the definition of $B_2^{(n+1)}, B_1^{(n)}$ we obtain
$$B_2^{(n+1)}-B_1^{(n)}=[R_h(\rnn)-I][u(r^{n+1})-u(r^n)]-\int_{r^n}^{r^{n+1}}\partial_r[R_h(r)u(r)-u(r)]dr.$$
Using Taylor's theorem, we obtain for $r_5,r_6\in(r^n,r^{n+1})$
$$u(r^{n+1})-u(r^n)=ku_r(\rnn)+\frac{k^3}{8\cdot
3!}[u_{rrr}(r_5)-u_{rrr}(r_6)],$$
The result now follows from the
estimates of $R_h(r)v(s)-v(s)$ and $\partial_r(R_h(r)v(r)-v(r))$.
\end{proof}
\begin{lemma}\label{elem2}For any $1\leq n\leq N-1$ it holds that
\begin{equation}\label{n9}
\|B_2^{(n+1)}-B_2^{(n)}\|\leq ck\Big{\{}h^{\tau+1}+k^2\Big{\}}.
\end{equation}
\end{lemma}
\begin{proof}
By the definition of $B_2^{(n+1)}, B_1^{(n)}$
we have that
\begin{equation}\label{n10}
\begin{split}
&B_2^{(n)}-B_2^{(n+1)}=R_h(\rnn)\Big{[}\frac{k^2}{8}u_{rr}(\rnn)-u(r^{n+1})\Big{]}\\
&-R_h(\rnnm)\Big{[}\frac{k^2}{8}u_{rr}(\rnnm)-u(r^{n})\Big{]}+R_h(r^{n+1})u(r^{n+1})-R_h(r^{n})u(r^{n}).
\end{split}
\end{equation}
Using Taylor's theorem we have that for
$r_7\in(\rnn,r^{n+1}),\;\;r_8\in(\rnnm,r^n)$
\begin{equation*}
\begin{split}
&u(r^{n+1})-\frac{k^2}{8}u_{rr}(\rnn)=u(\rnn)+\frac{k}{2}u_r(\rnn)+\frac{k^3}{8\cdot
3!}u_{rrr}(r_7),\\
&u(r^{n})-\frac{k^2}{8}u_{rr}(\rnnm)=u(\rnnm)+\frac{k}{2}u_r(\rnnm)+\frac{k^3}{8\cdot
3!}u_{rrr}(r_8).
\end{split}
\end{equation*}
Replacing these expansions  in \eqref{n10} it follows that
\begin{equation}\label{n11}
\begin{split}
&B_2^{(n)}-B_2^{(n+1)}=-R_h(\rnn)\Big{[}u(\rnn)+\frac{k}{2}u_r(\rnn)\Big{]}\\
&+R_h(\rnnm)\Big{[}u(\rnnm)+\frac{k}{2}u_r(\rnnm)\Big{]}+R_h(r^{n+1})u(r^{n+1})-R_h(r^{n})u(r^{n})+\mathcal{B}_1\\
&=\int_{r^n}^{r^{n+1}}[\partial_r
R_h(r)u(r)-u_r(r)]dr\\
&-\int_{\rnnm}^{\rnn}\Big{(}\partial_r
R_h(r)[u(r)+\frac{k}{2}u_r(r)]-[u_r(r)+\frac{k}{2}u_{rr}(r)]\Big{)}dr\\
&+u(r^{n+1})-u(r^n)-u(\rnn)+u(\rnnm)-\frac{k}{2}u_r(\rnn)+\frac{k}{2}u_r(\rnnm)+\mathcal{B}_1,
\end{split}
\end{equation}
where $|\mathcal{B}_1|\leq ck^3$ for $h<1$. Expanding in Taylor series around
$r^n,r^{n+1}$ we finally have that
$$|u(r^{n+1})-u(r^n)-u(\rnn)+u(\rnnm)-\frac{k}{2}u_r(\rnn)+\frac{k}{2}u_r(\rnnm)|\leq
c k^3,$$ and the result follows from \eqref{n11}.
\begin{remark}
Obviously we assumed $n\geq 1$, since we used the nodal point
$r^{n-\frac{1}{2}}$.
\end{remark}
\end{proof}
\begin{lemma}\label{elem3}
We have
\begin{equation}\label{n12}
\|B_1^{(0)}\|\leq ckh^{\tau+1}+ck^2,
\end{equation}
\begin{equation}\label{n13}
\|B_2^{(1)}\|\leq ckh^{\tau+1}+ck^2,
\end{equation}
\begin{equation}\label{n13*}
\|B_2^{(n+1)}\|\leq ch^{\tau+1}+ck^2,\;\;0\leq n\leq N-1.
\end{equation}
\end{lemma}
\begin{proof}
We use the fact that for $r_9\in(r^0,r^{\frac{1}{2}})$
$$u(r^0)=u(r^{\frac{1}{2}})-\frac{k}{2}u_r(r^{\frac{1}{2}})+\frac{k^2}{8}u_{rr}(r^{\frac{1}{2}})+\frac{k^3}{8\cdot
3!}u_{rrr}(r_9),$$ and obtain
\begin{equation*}
\begin{split}
B_1^{(0)}=&\int_{r^0}^{r^{\frac{1}{2}}}[\partial_rR_h(r)u(r)-u_r(r)]dr-\frac{k}{2}[R_h(r^{\frac{1}{2}})u_r(r^{\frac{1}{2}})
-u_r(r^{\frac{1}{2}})]\\
&+u(r^{\frac{1}{2}})-u(r^0)-\frac{k}{2}u_r(r^{\frac{1}{2}})+\mathcal{B}_2,
\end{split}
\end{equation*}
for $|\mathcal{B}_2|\leq c k^2$. Finally, using
$$|u(r^{\frac{1}{2}})-u(r^0)-\frac{k}{2}u_r(r^{\frac{1}{2}})|\leq ck^2,$$
we arrive at \eqref{n12}.

By Lemma \ref{elem1} applied for $n=0$, we get
$$\|B_2^{(1)}-B_1^{(0)}\|\leq ckh^{\tau+1},$$
so using \eqref{n12} the estimate \eqref{n13} follows.

Using that
$$\|B_2^{(n+1)}\|\leq \|B_2^{(n+1)}-B_2^{(n)}\|+\|B_2^{(n)}-B_2^{(n-1)}\|+\cdots+\|B_2^{(2)}-B_2^{(1)}\|+\|B_2^{(1)}\|,$$
and Lemma \ref{elem2} we obtain $$\|B_2^{(n+1)}\|\leq
\|B_2^{(1)}\|+ch^{\tau+1}+ck^2,$$ and by \eqref{n13} we arrive at
\eqref{n13*}.
\end{proof}
We now estimate $\theta^1$.
\begin{lemma}\label{elem4}
If $k=O(h)$ then
\begin{equation}\label{n14}
\|\theta^1\|\leq ch^{\tau+1}+ck^2.
\end{equation}
\end{lemma}
\begin{proof}
We use the continuous problem and the fact that $\theta^0=0$, set
$\phi=\theta^1$ in the fully discrete scheme, take real parts
and use the inverse inequality to obtain
\begin{equation*}
\begin{split}
&\Big{(}\frac{\theta^{n+1}-\theta^n}{k},\phi\Big{)}=-\Big{(}\frac{R_h(r^{n+1})u(r^{n+1})
-R_h(r^n)u(r^n)}{k}-u_r(r^{n+1/2}),\phi\Big{)}\\
&-{\rm
i}\mathcal{B}\Big{(}r^{n+1/2};\frac{\theta^{n+1}+\theta^n}{2},\phi\Big{)}+
{\rm
i}\Big{(}[\beta(r^{n+1/2})+\delta]\frac{\theta^{n+1}+\theta^n}{2},\phi\Big{)}\\
&-{\rm
i}\mathcal{B}\Big{(}r^{n+1/2};\frac{R_h(r^{n+1})u(r^{n+1})+R_h(r^n)u(r^n)}{2}-u(r^{n+1/2}),\phi\Big{)}\\
&+{\rm
i}\Big{(}[\beta(r^{n+1/2})+\delta]\frac{R_h(r^{n+1})u(r^{n+1})+R_h(r^n)u(r^n)}{2}-u(r^{n+1/2}),\phi\Big{)}.
\end{split}
\end{equation*}
Obviously, if $\mathcal{B}$ has smooth coefficients and $g,a,b$
are smooth, it follows that
\begin{equation*}
\begin{split}
&\Big{|}\mathcal{B}(r^{n+1/2};a,b)\Big{|}\leq
\frac{1}{2}\Big{|}\mathcal{B}(r^{n+1};a,b)+\mathcal{B}(r^{n};a,b)\Big{|}
+ck^2\|a\|_1\|b\|_1,\\
&\Big{|}\mathcal{B}(r^{n+1};a,b)-\mathcal{B}(r^{n};a,b)\Big{|}\leq
ck\|a\|_1\|b\|_1,\\
&\Big{|}\mathcal{B}(r;\frac{g(r^{n+1})+g(r^n)}{2},b)-\mathcal{B}(r;g(r^{n+1/2}),b)\Big{|}\leq
ck^2\|b\|_1.
\end{split}
\end{equation*}
So, we get
\begin{equation*}
\begin{split}
&\mathcal{B}\Big{(}r^{n+1/2};\frac{R_h(r^{n+1})u(r^{n+1})+R_h(r^n)u(r^n)}{2}-u(r^{n+1/2}),\phi\Big{)}=\\
&\mathcal{B}\Big{(}r^{n+1/2};\frac{R_h(r^{n+1})u(r^{n+1})+R_h(r^n)u(r^n)}{2}
-\frac{u(r^{n+1})+u(r^n)}{2},\phi\Big{)}+\mathcal{A}_1=\\
&\frac{1}{2}\mathcal{B}\Big{(}r^{n+1};\frac{R_h(r^{n+1})u(r^{n+1})+R_h(r^n)u(r^n)}{2}
-\frac{u(r^{n+1})+u(r^n)}{2},\phi\Big{)}\\
&+\frac{1}{2}\mathcal{B}\Big{(}r^{n};\frac{R_h(r^{n})u(r^{n+1})+R_h(r^n)u(r^n)}{2}
-\frac{u(r^{n+1})+u(r^n)}{2},\phi\Big{)}+\mathcal{A}_2+\mathcal{A}_1=\\
&\frac{1}{2}\mathcal{B}\Big{(}r^{n+1};\frac{R_h(r^n)u(r^n)}{2}
-\frac{u(r^n)}{2},\phi\Big{)}+\frac{1}{2}\mathcal{B}\Big{(}r^{n};\frac{R_h(r^{n+1})u(r^{n+1})}{2}
-\frac{u(r^{n+1})}{2},\phi\Big{)}\\
&+\mathcal{A}_2+\mathcal{A}_1=\\
&\frac{1}{2}\mathcal{B}\Big{(}r^{n};\frac{R_h(r^n)u(r^n)}{2}
-\frac{u(r^n)}{2},\phi\Big{)}+\mathcal{A}_3\\
&+\frac{1}{2}\mathcal{B}\Big{(}r^{n+1};\frac{R_h(r^{n+1})u(r^{n+1})}{2}
-\frac{u(r^{n+1})}{2},\phi\Big{)}+\mathcal{A}_4+\mathcal{A}_2+\mathcal{A}_1=\\
&\mathcal{A}_3 +\mathcal{A}_4+\mathcal{A}_2+\mathcal{A}_1,
\end{split}
\end{equation*}
where
\begin{equation*}
\begin{split}
&|\mathcal{A}_1|\leq ck^2\|\phi\|_1,\\
&|\mathcal{A}_2|\leq
ck^2\|\phi\|_1\;\;\;\;(\mbox{since}\;\;\|R_h(r)u(r)\|_1\leq
ch^{\tau}+c),\\
&|\mathcal{A}_3|,|\mathcal{A}_4|\leq
ckh^{\tau}\|\phi\|_1\;\;\;\;(\mbox{since}\;\;\|R_h(r)u(r)-u(r)\|_1\leq
ch^{\tau}).
\end{split}
\end{equation*}
Therefore, we obtain setting $n=0$, $\phi=\theta^1$ and using the
inverse inequality
\begin{equation*}
\begin{split}
\|\theta^1\|^2&\leq
ck[h^{\tau+1}+kh^{\tau}+k^2]\|\theta^1\|_1+ck\|\theta^1\|^2\\
&\leq
ckh^{-1}[h^{\tau+1}+kh^{\tau}+k^2]\|\theta^1\|+ck\|\theta^1\|^2.
\end{split}
\end{equation*}
So for $\mathcal{O}(k)=\mathcal{O}(h)$ the result follows.
\end{proof}
\begin{lemma}\label{elem6}
If $\mathcal{O}(k)=\mathcal{O}(h)$ then for any $n\geq 0$
\begin{equation}\label{n16}
\|\theta^{n+1}\|\leq ch^{\tau+1}+ck^2.
\end{equation}
\end{lemma}
\begin{proof}
Since $k\leq ch^{\frac{1}{2}}$ then the inequality \eqref{n8}
holds. So, using \eqref{n8} and Lemmas \ref{elem1}-\ref{elem2} we
arrive at
\begin{equation*}
\|\theta^{n+1}-B_2^{(n+1)}\|\leq
c\|\theta^n-B_2^{(n)}\|+ckh^{\tau+1}+ck^3,
\end{equation*}
therefore, setting $\mathcal{E}^{n}:=\theta^n-B_2^{(n)}$ for
$n\geq 1$, we obtain
\begin{equation*}
\begin{split}
\|\mathcal{E}^{n+1}\|&\leq
\Big{(}\frac{1+ck}{1-ck}\Big{)}\|\mathcal{E}^{n}\|+ckh^{\tau+1}+ck^3\\
&\leq
\Big{(}\frac{1+ck}{1-ck}\Big{)}^2\|\mathcal{E}^{n-1}\|+\Big{(}\frac{1+ck}{1-ck}\Big{)}(ckh^{\tau+1}+ck^3)+ckh^{\tau+1}+ck^3\\
&\leq\cdots\\
&\leq
 \Big{(}\frac{1+ck}{1-ck}\Big{)}^n\|\mathcal{E}^{1}\|+\Big{(}\frac{1+ck}{1-ck}\Big{)}^n\sum_{i=1}^n(ckh^{\tau+1}+ck^3)\\
&\leq c\|\mathcal{E}^{1}\|+ch^{\tau+1}+ck^2,
\end{split}
\end{equation*}
(since $k=\frac{1}{N}$ then
$\Big{(}\frac{1+ck}{1-ck}\Big{)}^N\rightarrow e^{c}$ as
$N\rightarrow \infty$ and thus
$\Big{(}\frac{1+ck}{1-ck}\Big{)}^n$ is bounded).

Thus, we get replacing $\mathcal{E}^{n+1},\;\mathcal{E}^{1}$
\begin{equation*}
\|\theta^{n+1}-B_2^{(n+1)}\|\leq
c\|\theta^1-B_2^{(1)}\|+ch^{\tau+1}+ck^2.
\end{equation*}
So we take
\begin{equation}\label{n17}
\|\theta^{n+1}\|-\|B_2^{(n+1)}\|\leq
c\|\theta^1-B_2^{(1)}\|+ch^{\tau+1}+ck^2.
\end{equation}
By \eqref{n13*} $\|B_2^{(n+1)}\|\leq ch^{\tau+1}+ck^2$, thus
\eqref{n17} together with \eqref{n13}, Lemma \ref{elem1} for
$n=0$, and Lemma \ref{elem4}, gives using the estimates of
$\|\theta^1\|$ and $\|B_1^{(0)}\|$
\begin{equation}\label{n18}
\begin{split}
\|\theta^{n+1}\|&\leq
c\|\theta^1-B_2^{(1)}\|+ch^{\tau+1}+ck^2\\
&\leq
c\|\theta^1-B_1^{(0)}\|+c\|B_1^{(0)}-B_2^{(1)}\|+ch^{\tau+1}+ck^2\\
&\leq c\|\theta^1\|+\|B_1^{(0)}\|+ch^{\tau+1}+ck^2\leq
ch^{\tau+1}+ck^2.
\end{split}
\end{equation}
\end{proof}
We are now ready to prove the main error estimate of this section:
\begin{theorem}\label{thmer}
If $\mathcal{O}(k)=\mathcal{O}(h)$ then
\begin{equation}\label{n19}
\|U^{n}-u(r^{n})\|\leq ch^{\tau+1}+ck^2,\;\;0\leq n\leq N.
\end{equation}
\end{theorem}
\begin{proof}
Obviously, using Lemma \ref{elem6} and the fact that
$\theta^0=0$, it follows that
$$\|U^{n}-u(r^{n})\|\leq \|\theta^{n}\|+ch^{\tau+1}\leq
ch^{\tau+1}+ck^2.$$
\end{proof}
%

\section{Global Elliptic Regularity}
%
In this section, we present a general Global Elliptic Regularity
Theorem for complex elliptic operators with mixed Dirichlet-Robin
boundary conditions, in rectangles of $\mathbb{R}^2$. Our proof
follows that of \cite{ref26} which deals with the Dirichlet
problem for real operators. In our approach, the main idea is
that if the trace terms in the weak formulation of the problem
vanish due to the boundary conditions, for suitably chosen test
functions, then a Global Elliptic Regularity result is proved in
Theorem \ref{3.2.9}. Note that the Robin condition in this
Theorem does not involve any zero order term, while the first
order terms are related to the coefficients of the boundary
problem  so that indeed in the weak formulation, after
integration by parts, the trace integrals vanish. Our result is
established by using the fact that the closure of a rectangle can
be covered by using a finite union of half-balls together with an
open smooth domain in the interior. We then apply  an exponential
transformation and extent our result, in Theorem \ref{3.2.10},
where an arbitrary zero order term is introduced at the Robin
condition of Theorem \ref{3.2.9}.
\begin{theorem}\label{3.2.9}
Let $\om=(0,1)\times (\az_1,\az_2)$ be a rectangular domain in
cartesian coordinates. We consider the following boundary value
problem: We seek a complex-valued function $u$ such that
\begin{equation}\label{1.p}
\begin{aligned}
&Au_{zz}+Bu_{z\az}+Cu_{\az\az}+Du_z+Eu_{\az}+Fu = f\;\;\;\mbox{ in }\;\;\; \om,\\
&u(0,\az) = 0, \\
&u(z,\az_1) = u(z,\az_2) = 0, \\
&a(\az)u_z+b(\az)u_{\az} = 0 \;\;\;\mbox{ at }\;\;\; z=1,
\end{aligned}
\end{equation}
%
%
where $A, B, C\in C^1(\overline{\om})$, $D, E, F\in
L^{\infty}(\om)$, $f\in L^2(\om)$ and $a, b:
[\az_1,\az_2]\rightarrow \mathbb{C}^*$. We also assume that $A,
B, C$ take imaginary values and $\frac{A}{{\rm i}}$,
$\frac{C}{{\rm i}}$ are always positive (or always negative).
Moreover, we assume that
\begin{gather}
|AC|>\frac{|B|^2}{4}, \qquad \text{for any } (z,\az)\in\om, \label{2.pa}\\
\frac{A(1,\az)}{a(\az)} = \frac{B(1,\az)}{2b(\az)},
\qquad\text{for any }\az\in [\az_1,\az_2]. \label{2.pb}
\end{gather}
If $u\in H^1(\om)$ is a weak solution of \eqref{1.p} then the
following elliptic regularity estimate holds
\begin{equation}\label{4p}
u\in H^2(\om)\quad{and}\quad \|u\|_{H^2(\om)}\le
c\,\|f\|_{L^2(\om)}.
\end{equation}
\end{theorem}
\begin{proof}
We consider the rectangle $\om$. Obviously its boundary is the
union of four linear segments and we write $\partial \om =
\cup_{i=1}^4\partial \om_i$ (cf. Figure 3). Let $\up_i =
B^o(k_i,r_i)\cap\overline{\om}$, be a half-ball in $\mathbb{R}^2$
in $\overline{\om}$ laying at $\partial\om$ of range $r_i$ and of
diameter in $\partial\om_i$. We define its boundary by
$\partial\up_i:=\partial\up_{ih}\cup\partial\up_{ic}$, where
$\partial\up_{ih}$ is the diameter such that $\partial\up
_{ih}\subseteq \partial \om_i$, and $\partial\up _{ic}$ is the
semicircle of range $r_i$ such that $\up_i\subset\overline{\om}$,
we also consider $\vp_i=B^o(k_i,r_i/2)\cap\overline{\om}$, the
half-ball being of the same center $k_i$ as $\up_i$ and of range
$r_i/2$ (cf. Figure 2). Obviously, $\partial{\om}$ is compact,
thus $\partial{\om}$ may be covered by using a finite union of
sets of the form $\vp_i$, while the same union together with a
suitably chosen smooth domain in $\om$ covers $\overline{\om}$. By
\cite{ref26} an interior regularity estimate holds. Our aim is to
prove the regularity estimate
%
\begin{equation}\label{6.p}
\|u\|_{H^2(\vp_i)}\le c\,\|f\|_{L^2(\up_i)},\quad i=1,\ldots,4.
\end{equation}
Interior regularity combined with the estimate (\ref{6.p}) gives
the desired result (\ref{4p}) (cf. \cite{ref26}, pg.~322).
\\
\begin{picture}(40,135)(0,0)
\put(90,30){\includegraphics[scale=.25,angle=0]{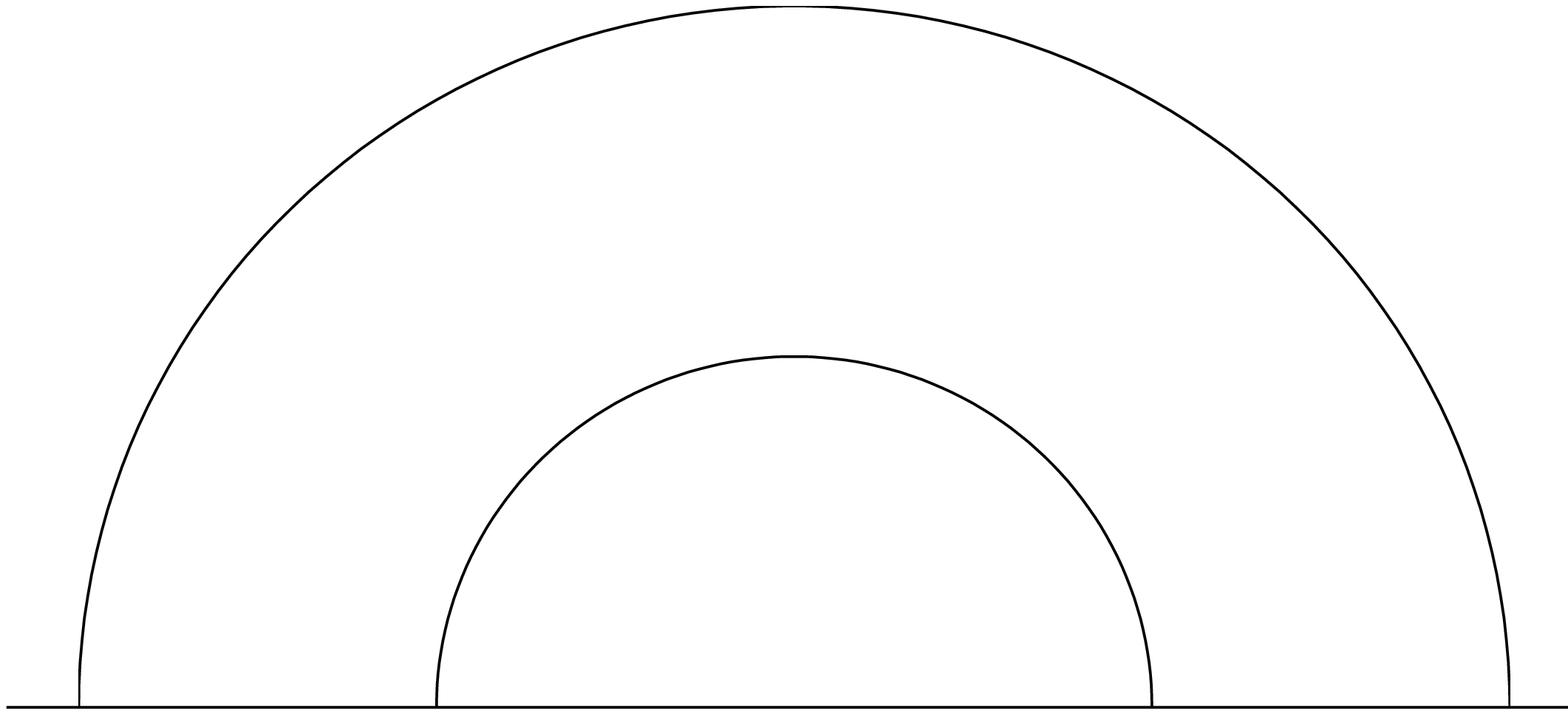}}
\put(55,0){\small{FIGURE 2. Half-balls, curved boundary,
horizontal
boundary.}}\put(95,112){$\mathcal{W}$}\put(210,105){$\partial \up
_{ic}$} \put(178,85){$\up _{i}$}\put(178,48){$\vp _{i}$}
\put(178,30){$\partial \up _{ih}$}\put(85,30){$\partial\mathcal{W}
_{i}$}
\end{picture}
\\

We consider $\phi_i\in H^1(\up_i)$ and let $u$ be the weak
solution of (\ref{1.p}). If $(u,v)_{\up_i}:=
\int_{\up_i}u\bar{v}\,ds$ then we have
\begin{equation}\label{7.p}
\begin{split} (f,\phi_i)_{\up_i}=&-(Au_z,\partial_z\phi_{i})_{\up_i}-
\Big{\{}\Big{(}\frac{B}{2}u_z,\partial_{\az}\phi_{i}\Big{)}_{\up_i}+
\Big{(}\frac{B}{2}u_{\az},\partial_z\phi_{i}\Big{)}_{\up_i}\Big{\}}\\
&-(Cu_{\az},\partial_{\az}\phi_{i})_{\up_i}+(\widetilde{D}u_{z},\phi_i)_{\up_i}+
(\widetilde{E}u_{\az},\phi_i)_{\up_i}+
(Fu,\phi_i)_{\up_i}\\
&+\int_{\partial\up_i}\Big{[}u_z
\Big{(}A,\frac{B}{2}\Big{)}+u_{\az}
\Big{(}\frac{B}{2},C\Big{)}\Big{]}\bar{\phi_i}\overrightarrow{\eta_{i}}ds,
\end{split}
\end{equation}
where
$\widetilde{D}$, $\widetilde{E}$ are the resulting terms after
integration by parts, and $\overrightarrow{\eta_{i}}$ is the
outward unit normal to $\partial\up_i$.
We let $\Omega_i(u,\phi_i):=\int_{\partial\up_i}[u_z
(A,\frac{B}{2})+u_{\az}
(\frac{B}{2},C)]\bar{\phi_i}\overrightarrow{\eta_{i}}ds$, and
define the vector $K_i:=[u_z (A,\frac{B}{2})+u_{\az}
(\frac{B}{2},C)]\bar{\phi_i}$; here $(\cdot,\cdot)$ denotes a
vector of $\mathbb{R}^2$. Then for
$\partial\up_{i}=\partial\up_{ih}\cup
\partial\up_{ic}$ it holds that $\Omega_i(u,\phi_i)=
\int_{\partial\up_{ih}}K_i\overrightarrow{\eta_{i}}ds
+\int_{\partial\up_{ic}}K_i\overrightarrow{\eta_{i}}ds$. Using
the boundary conditions of $u\in H^1(\om)$ we obtain
\begin{equation}\label{15.p}
\begin{aligned}
\Omega_1(u,\phi_1) &= -\int_{\partial\up_{1h}}Au_z\bar{\phi_1}ds+
\int_{\partial\up_{1c}}K_1\overrightarrow{\eta_{1}}ds,\;\;\;\;
 \Omega_2(u,\phi_2)& = \int_{\partial\up_{2c}}
 K_2\overrightarrow{\eta_{2}}ds,\\
\Omega_3(u,\phi_3) &=
-\int_{\partial\up_{3h}}Cu_{\az}\bar{\phi_3}ds+
\int_{\partial\up_{3c}}K_3\overrightarrow{\eta_{3}}ds,\\
\Omega_4(u,\phi_4) &=
\int_{\partial\up_{4h}}Cu_{\az}\bar{\phi_4}ds+
\int_{\partial\up_{4c}}K_4\overrightarrow{\eta_{4}}ds.
\end{aligned}
\end{equation}
Our aim now is to find test functions $\phi_i$ such that in the
weak formulation the trace terms vanish.
\subsubsection*{Assumption 1} We assume that there exist functions $\phi_i$ that satisfy the
following requirements:
\begin{itemize}
\item The test functions are smooth and along the curved boundary $\up_{ic}$ of  $\up_i$ vanish:
$\phi_i\in H^1(\up_i)$, and
$\phi_i=0|\partial\up_{ic},\;i=1,\dots,4$.
\item For $i=1,3,4$, the test functions vanish also along the horizontal boundary $\up_{ih}$ of $\up_i$:
$\phi_1=0$ at $z=0$, $\phi_2$ is arbitrary, $\phi_3=0$ at
$\az=\az_1$, $\phi_4=0$ at $\az=\az_2$.
\end{itemize}
Under this assumption, the sum of trace integrals in the weak
formulation equals zero because $\Omega_i(u,\phi_i)=0$ for any
$i=1,\ldots,4$. The weak formulation (\ref{7.p}) for
$\mathcal{B}(u,\phi_i)_{\up_i}:=(f,\phi_i)_{\up_i}$ becomes
\begin{equation}\label{18.p}
\begin{split} \mathcal{B}(u,\phi_i)_{\up_i}
 =&-(Au_z,\partial_z\phi_{i})_{\up_i}-
\Big{\{}(\frac{B}{2}u_z,\partial_{\az}\phi_{i}\Big{)}_{\up_i}+\Big{(}\frac{B}{2}u_{\az},\partial_z\phi_{i}
\Big{)}_{\up_i}\Big{\}}\\
&-(Cu_{\az},\partial_{\az}\phi_{i})_{\up_i}+(\widetilde{D}u_{z},\phi_i)_{\up_i}+
(\widetilde{E}u_{\az},\phi_i)_{\up_i}+ (Fu,\phi_i)_{\up_i}.
\end{split}
\end{equation}
%
\begin{picture}(30,148)(0,0)
\put(50,20){\includegraphics[scale=.35,angle=0]{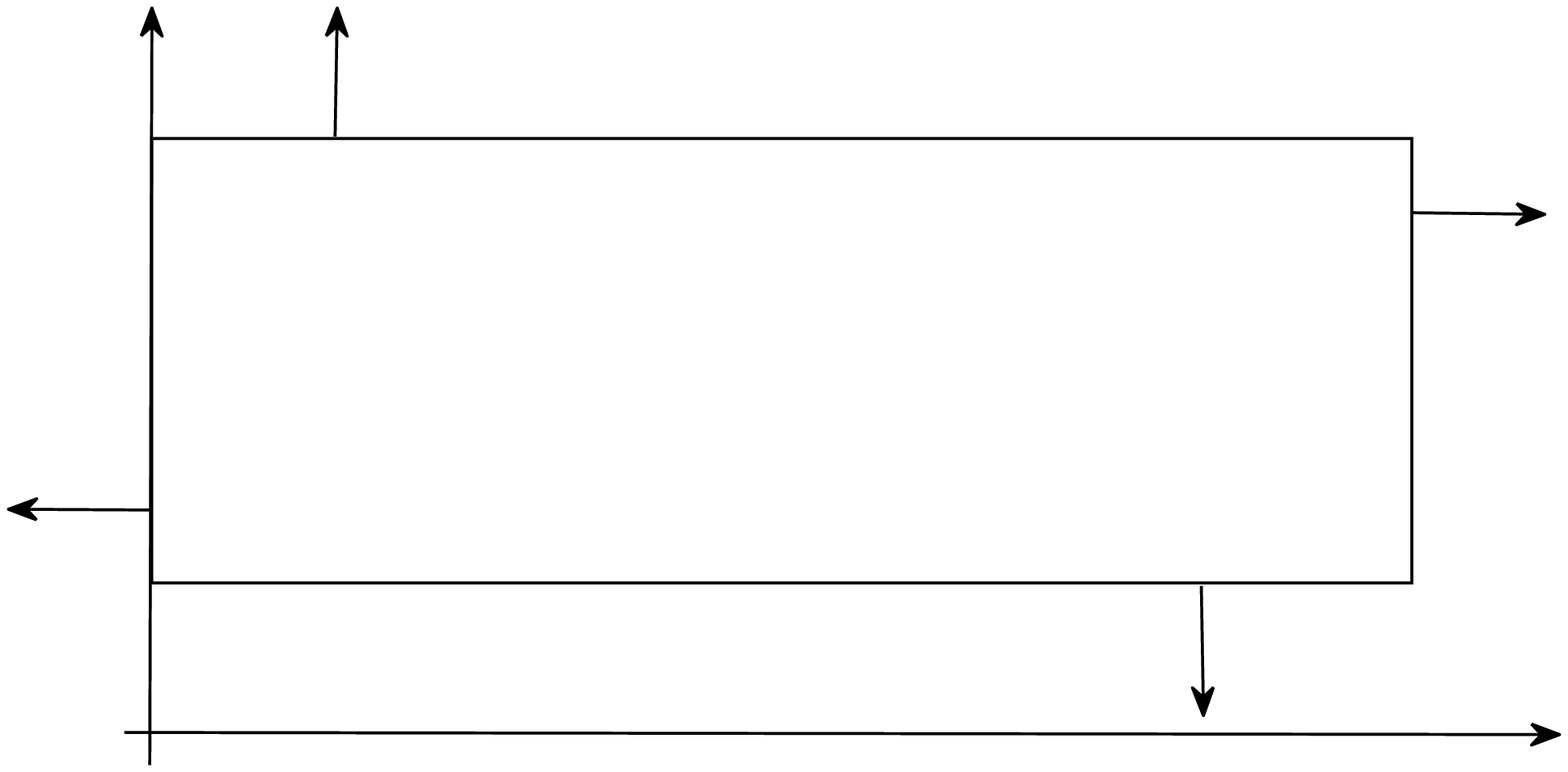}}
\put(90,33){$0$}\put(90,60){$\az_1$}\put(90,115){$\az_2$}
\put(173,118){$\partial \om_4$}
\put(105,125){$\az$}\put(270,33){$z$}\put(257,33){$1$}\put(173,50){$\partial
\om_3$}\put(80,85){$\partial
\om_1$}\put(175,85){$\om$}\put(260,85){$\partial \om_2$}
 \put(100,0){\small{ FIGURE 3. The rectangular domain $\om$.}}
\end{picture}
\\

The next step is to define, properly, for any $i=1,\cdots,4$, test
functions $\phi_i$ satisfying this assumption. We define the
following general cut-off function (\cite{ref26})
\begin{equation}\label{19.p}
 \left.\begin{array}{l}
 J=\begin{cases}
   0\mbox{ in }\mathbb{R}^2-\bl(\tilde{l},r),\\
   1\mbox{ in }\bl(\tilde{l},r/2),\\
   0\leq J\leq 1\mbox{ elsewhere (with $J=0$ near $\partial\up_c$)}.
  \end{cases}
\end{array}\right.
\end{equation}
Here $\up:=\bl^o(\tilde{l},r)\cap\overline{\om}$ is a half-ball in
$\mathbb{R}^2$ of radius $r$ and of center $\tilde{l}$ such that
$\partial\up_h\subseteq \vartheta\om$. Let $\vp$ be the half-ball
in $\mathbb{R}^2$ of center $\tilde{l}$ and of range $r/2$ with
diameter in $\partial\up_h$. Obviously the cut off function $J$
in $\vp$ equals $1$, and near $\partial\up_c$ is $0$. Let
$\tilde{u}$ be a function in $H^1(\om)$ that satisfies the
boundary conditions of problem (\ref{1.p}), we define the
function (\cite{ref26})
\begin{equation}\label{20.p}
 \left.\begin{array}{l}
\tilde{v}:=-D^{-h}(J^2D^h\tilde{u}),\mbox{  with
}D^h\tilde{u}(x):=\frac{\tilde{u}(x+he)-\tilde{u}(x)}{h},\;x\in\up,
\end{array}\right.
\end{equation}
where $h$ is a positive number and $e$ is a unitary vector
(direction) in $\mathbb{R}^2$ parallel to the diameter of the
half-ball $\up$.

In this way for every boundary line ($i=1,\cdots,4$) of the
rectangular domain $\om$ we define a cut-off function $J_i$ and
denote by $e_i$ the unitary direction of the specific boundary
line $\partial\om_i$. We then prove first that $\tilde{v}_i$
defined by these $J_i$ in \eqref{20.p} for the directions $e_i$
are test functions that satisfy the Assumption 1, and in the
sequel we set $\phi_i:=\tilde{v}_i$.

More specifically, for every $i=1,\ldots,4$ we consider
$\up_i=\bl^o(k_i,r_i)\cap
\overline{\om},\;\;\vp_i=\bl^0(k_i,\frac{r_i}{2})\cap\overline{\om},\;\;k_i,\;r_i$
such that $\up_i\subseteq\om,\;\partial\up_{ih}\subseteq\partial
\om_i$ and define the cut-off function
\begin{equation*}
 \left.\begin{array}{l}
J_i:=
  \begin{cases}
    J_i=0\mbox{ in
    }\mathbb{R}^2-\bl(k_i,r_i),\\
    J_i=1\mbox{ in }\bl(k_i,\frac{r_i}{2}),\\
    0\leq J_i\leq 1 \mbox{ elsewhere (with $J_i=0$ near $\partial\up_{ic}$)}.
  \end{cases}
\end{array}\right.
\end{equation*}
Let $\tilde{u}$ be a function in $H^1(\om)$ that satisfies the
boundary conditions of problem (\ref{1.p}), we define as
previously the function
\begin{equation}\label{20.pp}
 \left.\begin{array}{l}
\tilde{v}_i:=-D_i^{-h}(J_i^2D_i^h\tilde{u}),\mbox{  with
}D_i^h\tilde{u}(x):=\frac{\tilde{u}(x+he_i)-\tilde{u}(x)}{h},\;x\in\up_i.
\end{array}\right.
\end{equation}
By \cite{ref26}, for any $x\in \up_i$, the following identity
holds
\begin{equation}\label{.p}
 \left.\begin{array}{l}
\tilde{v}_i(x)=-\frac{1}{h^2}(J_i^2(x-he_{i})[\tilde{u}(x)-\tilde{u}(x-he_{i})]
-J_i^2(x)[\tilde{u}(x+he_{i})-\tilde{u}(x)]).
\end{array}\right.
\end{equation}

Using the boundary conditions of the elliptic problem and the
identity (\ref{.p}), we will prove
 that $\tilde{v}_i$ satisfy Assumption 1 for any $i=1,\ldots,4$.
\\
\\
\underline{If $i=1$}, then obviously $\tilde{v}_1$ is in
$H^1(\up_1)$. We notice that if $x$ is in $\partial\up_{1c}$ then
$J_1(x)=0$ and for $h$ small enough $J_1(x-he_1)=0$ so by
(\ref{.p}) $\tilde{v}_1(x)=0|\partial\up_{1c}$. Along the
boundary line $\partial\up_{1h}$ holds that $z=0$ and
$e_{1}=(0,1)$. If $x=(0,\az)$ then $\tilde{u}(x)=0$ and
$\tilde{u}(x\pm he_{1})=\tilde{u}(0,\az\pm h)=0$, thus by
(\ref{.p}) follows that $\tilde{v}_1(0,\az)=0$.
\\
\\
\underline{If $i=2$}, then $\tilde{v}_2(x)\in H^1(\up_2)$, and
$e_{2}=(0,1)$. If $x\in \partial\up_{2c}$ then for $h$ small
$J_2(x)=J_2(x-he_2)=0$, thus by (\ref{.p})
$\tilde{v}_2(x)=0|\partial\up_{2c}$.
\\
\\
\underline{If $i=3$}, then $\tilde{v}_3(x)\in H^1(\up_3)$ and
$e_{3}=(1,0)$, for $h$ small. If $x\in\up_{3c}$ then
$J_3(x)=J_3(x-he_3)=0$ thus $\tilde{v}_3(x)=0|\partial\up_{3c}$.
For $x=(z,\az_1)$ then $\tilde{u}(x)=\tilde{u}(z,\az_1)=0$ and
$\tilde{u}(x\pm he_{3})=\tilde{u}(z\pm h,\az_1)=0$. By (\ref{.p})
follows that $\tilde{v}_3(z,\az_1)=0$.
\\
\\
\underline{If $i=4$}, then $\tilde{v}_4(x)\in H^1(\up_4)$ and
$e_{4}=(1,0)$, if $x$ is in $\partial\up_{4c}$ then $J_4(x)=0$
and for $h$ small enough $J_4(x-he_4)=0$, thus by (\ref{.p})
$\tilde{v}_4(x)=0|\partial\up_{4c}$. If $x=(z,\az_2)$ then
$\tilde{u}(x)=\tilde{u}(z,\az_2)=0$ and $\tilde{u}(x\pm
he_{4})=\tilde{u}(z\pm h,\az_2)=0$, thus $\tilde{v}_4(z,\az_2)=0$.

Therefore, in all cases Assumption 1 holds and the
trace terms vanish from the weak formulation of the elliptic
problem. If we set $\tilde{u}:=u$, where $u$ is the weak solution
of the elliptic problem satisfying the boundary conditions,
then it can be easily proved (for details see \cite{refthes} and
\cite{ref26}) by use of ellipticity, the weak formulation and
the boundary conditions at $z=0$, $\az=\az_1$, $\az=\az_2$,
that for every half-ball $\vp_i$ it holds
\begin{equation}\label{.p27}
\|u\|_{H^2(\vp_i)}\leq c[\|f\|_{L^2(\up_i)}+\|u\|_{H^1(\up_i)}].
\end{equation}
Finite summation of (\ref{.p27}) over any $\vp_i$ (of type
$i=1,\cdots,4$) and the interior regularity give (\cite{ref26})
\begin{equation}\label{288}
\|u\|_{H^2(\om)}\le c[\|f\|_{L^2(\om)}+\|u\|_{H^1(\om)}].
\end{equation}
Combining (\ref{288}) with ellipticity we obtain the elliptic
regularity result
\begin{equation}\label{er}
\|u\|_{H^2(\om)}\leq c\|f\|_{L^2(\om)}.
\end{equation}
\end{proof}
\begin{remark}\label{gterm}
We note that an analogous result is also valid if in the
assumptions of Theorem \ref{3.2.9}, the homogeneous condition at
$z=1$ is replaced by the non-homogeneous condition
$a(\az)u_{z}+b(\az)u_\az=g$ at $z=1$, for any $g\in
H^{\frac{1}{2}}(\partial\om_R)$, where
$\partial\om_R=\{1\}\times(\az_1,\az_2)$. In this case, in the
weak formulation the trace integral term containing $g$ is hidden
due to ellipticity, leaving at the right-hand side of \eqref{er}
the extra term $c|g|_{\frac{1}{2},\partial\om_R}$ where
$|g|_{\frac{1}{2},\partial\om_R}:=\displaystyle{\inf_{v\in
H_{\om}:v|_{z=1}=g}}\|v\|_1$, for $H_{\om}:= \{u\in H^1(\om):
u|_{z=0} = 0\}$. More specifically, the following elliptic
regularity estimate holds
\begin{equation}\label{er*}
\|u\|_{H^2(\om)}\leq
c\|f\|_{L^2(\om)}+c|g|_{\frac{1}{2},\partial\om_R}.
\end{equation}
\end{remark}
The following theorem extends Theorem~\ref{1.p} in the sense that
we can add at the boundary condition along $z=1$ a zero order
term multiplied by an arbitrary smooth function $c(\az)$.
\begin{theorem}\label{3.2.10}
Under the assumptions of Theorem \ref{3.2.9}, if the boundary
condition of (\ref{1.p}) at $z=1$ has the form
\begin{equation}\label{**}
 \left.\begin{array}{l}
a(\az)u_z+b(\az)u_{\az}+c(\az)u(\az)=0\mbox{ at }z=1,\;\az \in
[\az_1,\az_2],
\end{array}\right.
\end{equation}
with $c$ a smooth complex function of $\az$, then the results of
Theorem \ref{3.2.9} hold (elliptic regularity).
\end{theorem}
\begin{proof}
We set $q=q(z,\az)$ and consider the elliptic operator of
(\ref{1.p}), we apply the transformation $u:=\exp(q)w$ and get
the following equivalent problem
\begin{equation}\label{35.p}
\begin{split}
 &Aw_{zz}+Bw_{z\az}+Cw_{\az\az}+D_ww_z+E_ww_{\az}+F_ww =f_w\;\;\;\mbox{ in }\;\;\;\mathcal{W},\\
&w(0,\az)=0,\\
&w(z,\az_1)=w(z,\az_2)=0,\\
&a(\az)w_z+b(\az)w_{\az}+c_w(\az)w(\az)=0\;\;\;\mbox{ at
}\;\;\;z=1,
\end{split}
\end{equation}
where $D_w=2Aq_z+Bq_{\az}+D$, $E_w=Bq_z+2Cq_{\az}+E$,
$f_w=\exp(-q)f$,
$F_w=F+A(q_{zz}+q_z^2)+B(q_{z\az}+q_zq_{\az})+C(q_{\az\az}+q_{\az}^2)+Dq_z+
Eq_{\az}$, and $c_w(\az)=a(\az)q_z+b(\az)q_{\az}+c(\az)$. We
chose $q(z,\az)$ such that $c_w(\az)=0$ or equivalently
\begin{equation}\label{36.p}
 \left.\begin{array}{l}
a(\az)q_z(1,\az)+b(\az)q_{\az}(1,\az)+c(\az)=0\;\mbox{ for any
}\az\in [\az_1,\az_2].
\end{array}\right.
\end{equation}
The relation (\ref{36.p}) can be achieved as
$\frac{a}{b}=\frac{2A}{B}$ is real, for $\frac{a(\az)}{b(\az)}$
smooth and $a(\az),b(\az)$ in $\mathbb{C}^*$, \cite{ref38}. Thus
by (\ref{35.p}) and (\ref{36.p}) the problem is of the form
covered by Theorem \ref{3.2.9}, and consequently
\begin{equation*}
 \left.\begin{array}{l}  w\in H^2(\om)\mbox{ and
}\|w\|_{H^2(\om)}\leq c\|f_w\|_{L^2(\om)}.
\end{array}\right.
\end{equation*}
Obviously $u=\exp(q)w$; therefore, $u\in H^2(\om)$ and
$\|u\|_{H^2(\om)}\leq c\|f\|_{L^2(\om)}$.
\end{proof}
\begin{remark}\label{gterm2}
By using Remark \ref{gterm}, under the assumptions of Theorem
\ref{3.2.10} and if we impose the non-homogeneous condition
$a(\az)u_z+b(\az)u_{\az}+c(\az)u(\az)=g$ at $z=1$, for $g \in
H^{\frac{1}{2}}(\partial\om_R)$ in place of the homogeneous one,
estimate \eqref{er*} follows (the proof is the same as in Theorem
\ref{3.2.10}).
\end{remark}
\begin{remark}
Theorem~\ref{3.2.9} and \ref{3.2.10} or the results of Remarks
\ref{gterm}, \ref{gterm2} can be applied to cylindrical
coordinates for $r$ fixed when $\om=\{(z,r,\az)\in
\mathbb{R}^3\}$, by use of the change of variables
$u(z,\az)=\hat{u}(z,\hat{\az})$ with $\hat{\az}:=\frac{2\pi
r\az}{360}=c_0\az$; then the equivalent problem in cartesian
coordinates is defined in a rectangular domain and satisfies the
assumptions of Theorems \ref{3.2.9} and \ref{3.2.10} or those of
Remarks \ref{gterm}, \ref{gterm2}.
\end{remark}

\section{Numerical experiments}
In this section we report on the outcome of some numerical
experiments performed with the fully discrete scheme
\eqref{3.101} to solve the initial-{} and boundary-value problem
\eqref{divgen}. In the notation established in Section~1, cf.
\eqref{divgen}, we took $\mathfrak{D} = (0,1)^2$,
$r_{\mathrm{min}} = 0$, $r_{\mathrm{max}} = 1$, $b = 0$,
$\beta=1$, $D$ the identity matrix, $\lambda = (0,1)$ and
right-hand side $F$ so that the exact solution is
\begin{equation}
u(r,y,\theta) = e^{2r} y (e^{-y}-1) \theta (1-\theta)^3.
\end{equation}

Our first set of experiments concerns the experimental
verification of the convergence rate of the scheme in the spatial
variable. The measure of the error was the $E(r) = \|u - U\|$ for
$r=nk$, $n=1,2,\ldots$, whereas for other values of $r$ $E$ was
defined by linear interpolation. To determine experimentally the
spatial order of convergence the approximate solution was
computed for $0\le r\le 1$ using a rectangular partition of
$\mathfrak{D}$ using $N = h^{-1}$ ranging from 20 to 160. The
finite element space $S_h$ consisted of piecewise polynomial
functions of degree one. For these runs, very small  $r$-steps
were taken to ensure that the error due to the discretization in
time-like variable $r$ is negligible. The observed error was
recorded at $r = 0.1, 0.5$ and $1$. As usual, the convergence
rate corresponding to two different runs with mesh sizes $h_1,
h_2$ and corresponding errors $E_1$ and $E_2$ is defined to be
$\log(E_1/E_2) / \log(h_1/h_2)$. The results are shown in
Table~\ref{scr1}. It is evident that the convergence rate of the
spatial component of the error is indeed two.

The determination of the accuracy in the time-like variable $r$
is more delicate. We took $h^{-1}=20$ and computed the solution
of our problem up to $r=1$ for various values of $k$. For this
fixed value of $h$ we made a reference calculation with a small
value of $k=k_{\mathrm{ref}} = h / 30$. The corresponding
approximate solution, denoted by $U_{h,\mathrm{ref}}$ differs
from the exact solution by a factor which is almost entirely due
to the spatial discretization. We then define a modified measure
of the error $E^{*}(r)$ as above but with the exact solution
replaced by  the reference solution $U_{h,\mathrm{ref}}$. The
results are shown in Table~\ref{tcr}.
\begin{table}
\caption{Errors $E(r)$ and spatial convergence rate for $k^{-1} =
400$} \centering \label{scr1}
\begin{tabular}{r|cc|cc|cc} \hline
& \multicolumn{2}{c|}{$r = 0.1$} & \multicolumn{2}{c|}{$r =
0.5$}& \multicolumn{2}{c}{$r = 1.0$}\\ \hline $h^{-1}$ & $E(r)$ &
Rate & $E(r)$         & Rate & $E(r)$         & Rate \\ \hline
  10  & 3.5162(-2)  &         & 4.9653(-2) &         & 7.8266(-2) &      \\ \hline
  20  & 7.5323(-3)  & 2.22 & 1.0734(-2) & 2.21 & 1.6921(-2) & 2.21 \\ \hline
  40  & 1.7219(-3)  & 2.13 & 2.4518(-3) & 2.13 & 3.8920(-3) & 2.12 \\ \hline
  80  & 4.0438(-4)  & 2.09 & 5.7998(-4) & 2.08 & 9.2042(-4) & 2.08 \\ \hline
160  & 9.7655(-5)  & 2.05 & 1.4100(-4) & 2.04 & 2.2381(-4) & 2.04
\\ \hline
\end{tabular}
\end{table}
\begin{table}[h]
\caption{Errors $E(r)$ and  $r$-convergence rate for $h^{-1} =
20$} \centering \label{tcr}
\begin{tabular}{r|ccc} \hline
$k^{-1}$ & $E(r)$     & $E^{*}(r)$   & Rate \\ \hline
  144   & 3.8104(-1) & 3.9217(-1) &          \\ \hline
  192   & 7.1839(-1) & 1.7832(-1) & 2.74 \\ \hline
  240   & 1.1771(-2) & 1.0652(-1) & 2.31 \\ \hline
  288   & 1.7638(-2) & 7.1442(-2) & 2.19  \\ \hline
  600   & 8.9952(-3) &                   &         \\ \hline
\end{tabular}
\end{table}

\section*{Acknowledgments}
D.~C. Antonopoulou acknowledges the support of the National
Scholarship Foundation of Greece (Postdoctoral Research in
Greece) and her advisor Prof. V.~A. Dougalis for proposing this
problem that was partially analyzed in her Ph.D. Thesis. G.~D.
Karali is supported by a Marie Curie International Reintegration
Grant within the 7th European Community Framework Programme,
MIRG-CT-2007-200526. G.~D. Karali and M. Plexousakis are
partially supported by the FP7-REGPOT-2009-1 project `Archimedes
Center for Modeling, Analysis and Computation'.
\bibliographystyle{amsplain}

\begin{thebibliography}{10}
%
\bibitem{ref1}
L. Abrahamsson, H.~O. Kreiss, \textit{The initial boundary value
problem for the Schr\"odinger equation}, Math. Methods Appl. Sci.
\textbf{13} (1990), 385--390.
%
\bibitem{ref2}
L. Abrahamsson, H.~O. Kreiss, \textit{Boundary conditions for the
parabolic equation in a range-dependent duct}, J. Acoust. Soc.
Amer. \textbf{87} (1990), 2438--2441.
%
\bibitem{AD}
G.~D. Akrivis, V.~A. Dougalis, \textit{Finite difference
discretization with variable mesh of the Schr\"odinger equation
in a variable domain}, Bull. Greek Math. Soc. \textbf{31} (1990),
19--28.

\bibitem{AkDZ1996}
G.~D. Akrivis, V.~A. Dougalis and G.~E. Zouraris, \textit{Error
estimates for finite difference methods for a wide--angle
`parabolic' equation}, SIAM J. Numer. Anal. \textbf{33} (1996),
2488--2509.

\bibitem{ref8}
G.~D. Akrivis, V.~A. Dougalis and G.~E. Zouraris, \textit{Finite
difference schemes for the `Parabolic' Equation in a variable
depth environment with a rigid bottom boundary condition}, SIAM J.
Numer. Anal. \textbf{39} (2001), 539--565.

%
\bibitem{refthes}
D.~C. Antonopoulou, Theory and Numerical Analysis of Parabolic
Approximations, Ph.D. Thesis, University of Athens, 2006 (in
Greek).
%

\bibitem{AnDSZ2008}
D.~C. Antonopoulou, V.~A. Dougalis, F. Sturm and G.~E. Zouraris
{\em Conservative initial-boundary value problems for the
wide-angle PE in waveguides with variable bottoms}, Proceedings
of the 9th European Conference on Underwater Acoustics (9th
EQUA), M.~E. Zakharia, D. Cassereau and F. Lupp{\'e}, eds.
\textbf{1}, 375--380 (2008).

\bibitem{AnDZ}
D.~C. Antonopoulou, V.~A. Dougalis and G.~E. Zouraris,
\textit{Galerkin Methods for Parabolic and Schr\"odinger Equations
with dynamical boundary conditions and applications to underwater
acoustics}, SIAM J. Numer. Anal., \textbf{47} (2009), 2752--2781.
%
%
%
\bibitem{An-P}
D.~C. Antonopoulou, M. Plexousakis, \textit{Discontinuous Galerkin
methods for the linear Schr\"o\-dinger equation in non-cylindrical
domains}, Numer. Math. \textbf{115} (2010), 585--608.

\bibitem{ref12}
A. Bamberger, B. Engquist, L. Halpern, P. Joly, \textit{Parabolic
wave equation approximations in heterogeneous media}, SIAM J.
Appl. Math. \textbf{48} (1988), 99--128.
%
\bibitem{ref14}
S. C. Brenner and L. R. Scott, \textit{The Mathematical Theory of
Finite Element Methods}, Springer-Verlag, New York, 1994.
\bibitem{ref16}
M.~J. Buckingham, \textit{Theory of three-dimensional acoustic
propagation in a wedge-like ocean with a penetrable bottom}, J.
Acoust. Soc. Amer. \textbf{82} (1987), 198--210.
%
\bibitem{stcas}
K. Castor, F. Sturm, \textit{Investigation of 3D acoustical
effects using a multiprocessing parabolic equation based
algorithm}, J. Comput. Acoust. \textbf{16(2)} (2008), 137--162.
%
\bibitem{ref21}
M.~D. Collins, S.~A. Chin-Bing, \textit{A three-dimensional
parabolic equation model that includes the effects of rough
boundaries}, J. Acoust. Soc. Amer. \textbf{87} (1990),
1104--1109.
%
\bibitem{ref23}
G.~B. Dean, M.~J. Buckingham, \textit{An analysis of the
three-dimensional sound field in a penetrable wedge with a
stratified fluid or elastic basement}, J. Acoust. Soc. Amer.
\textbf{93} (1993), 1319--1328.
%
%
\bibitem{DSZ2009}
V.~A. Dougalis, F. Sturm and G.~E. Zouraris, \textit{On an
initial-boundary value problem for a wide-angle parabolic
equation in a waveguide with a variable bottom}, Math. Meth. in
Appl. Sciences \textbf{32} (2009), 1519--1540.
%
%
\bibitem{Dupont1973}
T.~Dupont,
\textit{$L^2 $-Estimates for Galerkin Methods for Second Order Hyperbolic Equations},
SIAM J. Numer. Anal. \textbf{10} (1973), 880--889.
%
%
\bibitem{ref26}
L.~C. Evans, \textit{Partial Differential Equations}, American
Mathematical Society, 1998.
%
\bibitem{ref33}
J.~A. Fawcett, \textit{Modeling three-dimensional propagation in
an oceanic wedge using parabolic equation methods}, J. Acoust.
Soc. Amer. \textbf{93(5)} (1993), 2627--2632
%
\bibitem{ref34}
R.~W. Freund, N.~M. Nachtigal, \textit{QMR: a quasi-minimal
residual method for non-Hermitian linear systems}, Numer. Math.
\textbf{60} (1991), 315--339.
%
\bibitem{jensenferla}
F.~B. Jensen, C.~M. Ferla, \textit{Numerical solutions of range-dependent benchmark problems
in ocean acoustics}, J. Acoust. Soc. Am. \textbf{87} (1990), 1499--1510.

\bibitem{ref38}
F. John, \textit{Partial Differential Equations},
Springer-Verlag, New York, 1982.
%
\bibitem{ref46}
J.~L. Lions, E. Mag\'enes, \textit{Probl\`emes aux Limites Non
Homog\`enes et Applications, I}, Dunod, Paris, 1968.
%
\bibitem{ref53}
F. Sturm, \textit{Mod\'elisation math\'ematique et num\'erique d'
un probl\`eme de propagation en acoustique sous-marine: prise en
compte d'un environnement variable tridimensionnel}, Th\`ese de
Docteur en Sciences Universit\'e de Toulon et du Var, France,
1997.
%
\bibitem{ref54}
F.~D. Tappert, \textit{The parabolic approximation method, Wave
Propagation and Under\-water Acoustics}, J.B. Keller and J.S.
Papadakis, eds., Lecture Notes in Phys. \textbf{70},
Springer-Verlag, Berlin (1977), 224--287.
%
\bibitem{ref55}
V. Thom\'ee, \textit{Galerkin Finite Element Methods for Parabolic
Problems}, Springer--Verlag, Berlin, 1997.
%
\bibitem{ref56}
D.~E. Weston, \textit{Horizontal refraction in a three
dimensional medium of variable stratification}, Proc. Roy. Soc.
London \textbf{78} (1961), 46--52.
%
%
\end{thebibliography}
%
\end{document}